\documentclass[11pt]{amsart}
\usepackage{amsfonts,amssymb,amsthm}
\usepackage{amsmath,amscd}
\usepackage{pstricks}
\usepackage{pstricks,pst-node}
\usepackage{mathrsfs}
\usepackage[all]{xy}

\DeclareMathAlphabet{\mathpzc}{OT1}{pzc}{m}{it}

\renewcommand{\subsection}[1]{\vspace{.18in}
\par\noindent\addtocounter{subsection}{1}
\setcounter{equation}{0}{\bf\thesubsection.\hspace{5pt}#1}}

\theoremstyle{definition}

\theoremstyle{plain}
\newtheorem{Prop}[subsection]{Proposition}
\newtheorem{Thm}[subsection]{Theorem}
\newtheorem{Not}[subsection]{Notation}
\newtheorem{Lem}[subsection]{Lemma}
\newtheorem{Coro}[subsection]{Corollary}
\numberwithin{equation}{subsection}

\newcommand{\tA}{{}^t\!A}
\newcommand{\tB}{{}^t\!B}
\newcommand{\tC}{{}^t\!C}
\newcommand{\tX}{{}^t\!X}
\newcommand{\tY}{{}^t\!Y}
\newcommand{\tw}{{}^t\!w}

\newcommand{\han}{\subseteq}

\newcommand{\bse}{\boldsymbol{e}}
\newcommand{\bsa}{{\mathbf{a}}}
\newcommand{\bsb}{\mathbf{b}}

\newcommand{\lan}{\langle}
\newcommand{\ran}{\rangle}
\newcommand{\dleb}{\left[\!\!\left[}
\newcommand{\leb}{\left[}
\newcommand{\drib}{\right]\!\!\right]}
\newcommand{\rib}{\right]}
\def\lr#1{\langle #1\rangle}
\def\br#1{\{ #1 \}}
\def\bpa#1#2{\left({#1\atop #2}\right)}
\def\ddet#1{|\!| #1 |\!|}

\def\ggp#1#2{\left[\kern-3.2pt\left[{#1\atop #2}\right]\kern-3.2pt\right]}

\def\fS{{\frak S}}

\def\fB{{\frak B}}
\def\fC{{\frak C}}
\def\fka{{\frak a}}

\newcommand{\msX}{\mathpzc X}
\newcommand{\msY}{\mathpzc Y}
\newcommand{\msZ}{\mathpzc Z}

\newcommand{\msD}{\mathscr D}\newcommand{\afmsD}{{\mathscr D}^\vtg}

\newcommand{\affSr}{{\fS_{\vtg,r}}}

\newcommand{\afHr}{{\sH_\vtg(r)}}

\def\sH{{\mathcal H}}

\def\sK{{\mathcal K}_\mbz(n)}
\def\sKq{{\mathcal K}_\mbq(n)}
\def\hsKq{\widehat{\mathcal K}_\mbq(n)}
\def\tisK{\ti{\mathcal K}(n)}
\def\sL{{\mathcal L}}
\def\sM{{\mathcal M}}

\def\sO{{\mathcal O}}

\def\sS{{\mathcal S}}

\def\sU{{\mathcal U}}
\def\sV{{\mathcal V}}
\def\sW{{\mathcal W}}

\def\sZ{{\mathcal Z}}

\newcommand{\vtg}{{\!\vartriangle\!}}
\newcommand{\Hall}{{{\mathfrak H}_\vtg(n)}}
\newcommand{\Hallz}{{{\mathfrak H}_\vtg(n)_\mbz}}
\newcommand{\Hallq}{{{\mathfrak H}_\vtg(n)_\mbq}}
\newcommand{\Hallzop}{{{\mathfrak H}_\vtg(n)_\mbz^{\mathrm{op}}}}

\newcommand{\bfHall}{{\boldsymbol{\mathfrak H}_\vtg(n)}}
\newcommand{\Hallpi}{\bfHall^{\geq 0}}
\newcommand{\Hallp}{\Hall^{\geq 0}}

\def\field{{\mathbb F}}

\def\rep{{\rm {\bf Rep}}}

\newcommand{\mnmod}{\!\!\!\mod\!}

\newcommand{\tri}{\triangle(n)}

\newcommand{\Bz}{\fB}
\newcommand{\Vz}{\sV_\mbz(n)}
\newcommand{\Vq}{\sV_\mbq(n)}

\newcommand{\Vqp}{\sV^+_\mbq(n)}
\newcommand{\Vqm}{\sV^-_\mbq(n)}
\newcommand{\Vqz}{\sV^0_\mbq(n)}

\newcommand{\tiVzp}{\ti\sV^+_\mbz(n)}
\newcommand{\Vzp}{\sV^+_\mbz(n)}
\newcommand{\Vzm}{\sV^-_\mbz(n)}
\newcommand{\Vzz}{\sV^0_\mbz(n)}

\newcommand{\afgl}{\widehat{\frak{gl}}_n}
\newcommand{\afuglq}{\sU(\widehat{\frak{gl}}_n)}
\newcommand{\afuglk}{\sU_k(\widehat{\frak{gl}}_n)}
\newcommand{\afuglqz}{\sU^0(\widehat{\frak{gl}}_n)}
\newcommand{\afuglqp}{\sU^+(\widehat{\frak{gl}}_n)}
\newcommand{\afuglqpz}{\sU^{\geq 0}(\widehat{\frak{gl}}_n)}
\newcommand{\afuglqm}{\sU^-(\widehat{\frak{gl}}_n)}
\newcommand{\afuglz}{\sU_\mbz(\widehat{\frak{gl}}_n)}
\newcommand{\afuglzp}{\sU_\mbz^+(\widehat{\frak{gl}}_n)}
\newcommand{\afuglzm}{\sU_\mbz^-(\widehat{\frak{gl}}_n)}
\newcommand{\afuglzz}{\sU_\mbz^0(\widehat{\frak{gl}}_n)}

\newcommand{\afg}{\mathfrak n}
\newcommand{\afFn}{{\mathscr F_\vtg}}

\newcommand{\afE}{E^\vartriangle}

\newcommand{\cycn}{{{n}}}
\newcommand{\afSr}{{\mathcal S}_{\vtg}(\cycn,r)}
\newcommand{\afSrk}{{\mathcal S}_{\vtg}(\cycn,r)_{k}}
\newcommand{\afSrmbz}{{\mathcal S}_{\vtg}(\cycn,r)_{\mathbb Z}}
\newcommand{\afSrmbq}{{\mathcal S}_{\vtg}(\cycn,r)_{\mathbb Q}}

\newcommand{\afbse}{\boldsymbol e^\vartriangle}
\newcommand{\afmbnn}{\mathbb N_\vtg^{\cycn}}
\newcommand{\afmbzn}{\mathbb Z_\vtg^{\cycn}}

\newcommand{\afLa}{\Lambda_\vtg}
\newcommand{\afLanr}{\Lambda_\vtg(\cycn,r)}

\newcommand{\afThn}{\Theta_\vtg(\cycn)}
\newcommand{\aftiThn}{\widetilde\Theta_\vtg(\cycn)}

\newcommand{\afThnpm}{\Theta_\vtg^\pm(\cycn)}
\newcommand{\afThnp}{\Theta_\vtg^+(\cycn)}
\newcommand{\afThnm}{\Theta_\vtg^-(\cycn)}

\newcommand{\afThnr}{\Theta_\vtg(\cycn,r)}
\newcommand{\afTh}{\Theta_\vtg}
\newcommand{\afMnq}{M_{\vtg,\cycn}(\mathbb Q)}
\newcommand{\afMnz}{M_{\vtg,\cycn}(\mathbb Z)}

\newcommand{\thp}{\th^+}
\newcommand{\thm}{\th^-}
\newcommand{\etap}{\eta^+}
\newcommand{\etarp}{\eta_r^+}
\newcommand{\etarm}{\eta_r^-}
\newcommand{\dzr}{\dot{\zeta}_r}
\newcommand{\hzr}{\h{\zeta}_r}
\newcommand{\hz}{\h{\zeta}}

\newcommand{\ttm}{\mathtt{m}}
\newcommand{\ttn}{\mathtt{n}}
\def\leq{\leqslant}\def\geq{\geqslant}
\def\le{\leqslant}\def\ge{\geqslant}

\newcommand{\dt}{\delta}

\newcommand{\Dt}{\Delta}
\newcommand{\lm}{\longmapsto}
\newcommand{\map}{\mapsto}

\newcommand{\og}{\omega}
\newcommand{\tong}{\stackrel{\thicksim}{\,\rightarrow}}
\newcommand{\vi}{\varphi}

\newcommand{\up}{v}
 \newcommand{\nup}{v}
 \newcommand{\ep}{\varepsilon}
 \newcommand{\al}{\alpha}
 \newcommand{\bt}{\beta}
 \newcommand{\h}{\widehat}
 \newcommand{\ti}{\widetilde}
\newcommand{\zr}{\zeta_r}

\newcommand{\sg}{\sigma}
\newcommand{\Sg}{\Sigma}
\def\th{\theta}

\newcommand{\p}{\prec}
\newcommand{\pr}{\preccurlyeq}
\newcommand{\bop}{\bigoplus}
\newcommand{\op}{\oplus}
\newcommand{\ot}{\otimes}

\newcommand{\bfl}{\mathbf{0}}

\newcommand{\ol}{\overline}
\newcommand{\lra}{\longrightarrow}
\newcommand{\ra}{\rightarrow}
 \newcommand{\la}{{\lambda}}

 \newcommand{\La}{\Lambda}

 \newcommand{\mbn}{\mathbb N}
 \newcommand{\mbq}{\mathbb Q}
 \newcommand{\mbc}{\mathbb C}
 \newcommand{\mbz}{\mathbb Z}

  \newcommand{\bfd}{{\mathbf{d}}}
 
 \newcommand{\bfj}{{\mathbf{j}}}

\newcommand{\bfa}{{\mathbf{a}}}
\newcommand{\bfb}{{\mathbf{b}}}

\newcommand{\bfL}{{\mathbf{L}}}

\newcommand{\ga}{{\gamma}}

\newcommand{\Aut}{\operatorname{Aut}}
\newcommand{\End}{\operatorname{End}}

\newcommand{\spann}{\operatorname{span}}
\newcommand{\diag}{\operatorname{diag}}

\def\ro{\text{\rm ro}}
\def\co{\text{\rm co}}

\newcommand{\ul}{\underline}
\newcommand{\bfsg}{{\boldsymbol\sigma}}

\def\afsygr{{\fS_{\vtg,r}}}

\begin{document}
\title{BLM realization for $\sU_\mbz(\h{\frak{gl}}_n)$}
\author{Qiang Fu}
\address{Department of Mathematics, Tongji University, Shanghai, 200092, China.}
\email{q.fu@hotmail.com}


\thanks{Supported by the National Natural Science Foundation
of China, the Program NCET, Fok Ying Tung Education Foundation
and the Fundamental Research Funds for the Central Universities}

\begin{abstract}
In 1990, Beilinson--Lusztig--MacPherson (BLM) discovered a realization \cite[5.7]{BLM} for
quantum $\frak{gl}_n$ via a geometric setting of quantum Schur
algebras. We will generailze their result to the classical affine case. More precisely, we first use Ringel--Hall algebras to construct an integral form $\afuglz$ of $\afuglq$, where $\afuglq$ is
the universal enveloping algebra of the loop algebra $\afgl:=\frak{gl}_n(\mbq)\ot\mbq[t,t^{-1}]$. We then establish the stabilization property of multiplication for the classical affine Schur algebras. This stabilization property leads to the BLM realization of $\afuglq$ and $\afuglz$. In particular, we conclude that  $\afuglz$ is a $\mbz$-Hopf subalgebra of $\afuglq$. As a bonus, this method leads to an explicit $\mbz$-basis for $\afuglz$, and it yields explicit multiplication formulas between generators and basis elements for $\afuglz$. As an application, we will
prove that the natural algebra homomorphism from $\sU_\mbz(\h{\frak{gl}}_n)$ to the affine Schur algebra over $\mbz$ is surjective.
\end{abstract}
 \sloppy \maketitle
\section{Introduction}
The positive part  of the integral form of quantum enveloping algebras of finite type was realized by Ringel, see \cite{R90,R932}. This is
an important breakthrough for the structure of quantum groups. Almost at the same time,   A.A. Beilinson, G. Lusztig and R.
MacPherson  \cite[5.7]{BLM} realized the entire quantum $\frak{gl}_n$ over the rational function field $\mbq(\up)$ (with $\up$ being an indeterminate) via quantum Schur algebras. It is natural to ask how to realize the integral form of the entire quantum $\frak{gl}_n$. If this can be achieved, then one can relaize quantum $\frak{gl}_n$ over any field.
The remarkable BLM's work has important applications to the investigation of quantum Schur-Weyl reciprocity. The classical Schur-Weyl reciprocity relates representations of the general linear and symmetric groups over $\mbc$(cf. \cite{Weyl}). This reciprocity is also true over any field (cf. \cite{CL,CP,Donkin}). The quantum Schur-Weyl reciprocity at nonroots of unity was first formulated in \cite{Jimbo}. Using BLM's work, the integral quantum Schur-Weyl reciprocity was established in \cite{Du95,DPS}.

The BLM realization problem of quantum affine $\frak{gl}_n$ was investigated in \cite{DF09,DDF}. In particular, it was proven that the natural algebra homomorphism from quantum affine $\frak{gl}_n$ to affine  quantum Schur algebras over $\mbq(\up)$ is surjective in \cite{DDF} (cf. \cite{GV,Lu99}). Furthermore, the universal enveloping algebra $\afuglq$ of $\afgl$  was realized in \cite{DDF} using a modified BLM approach. However, in the affine case, there are still many important unsolved problems. For example, the stabilization property of multiplication for quantum Schur algebras given in \cite[4.2]{BLM} is the key to the BLM realization of quantum $\frak{gl}_n$. Furthermore, explicit multiplication formulas between generators and basis elements for the quantum enveloping algebra of $\frak{gl}_n$ were obtained in \cite[5.3]{BLM}.
But it seems hard to generalize these results to the quantum affine case. In addition, it is difficult to construct a suitable integral form for quantum affine $\frak{gl}_n$ such that the integral quantum affine Schur reciprocity holds (cf. \cite[3.8.6]{DDF}).

In this paper, we will solve the above problems in the classical case. First, we will use Ringel--Hall algebras to construct a free $\mbz$-submodule $\afuglz$ of the universal enveloping algebra $\afuglq$ of the loop algebra $\afgl$ in \S3.  We then
prove in \ref{stabilization property} a stabilization property for the structure constants of an affine Schur algebra, which is the affine analogue of \cite[4.2]{BLM}.  This  property allows us to construct an algebra $\sK$ without unity. Then we consider the completion algebra $\hsKq$ of $\sK$ and construct a $\mbz$-submodule $\Vz$ of $\hsKq$. We will prove in \ref{formula in Vz} and \ref{Vz-subalgebra} that $\Vz$ is a $\mbz$-subalgebra of $\hsKq$ with nice multiplication formulas, which is the affine analogue of \cite[5.3 and 5.5]{BLM}. Finally,
we will prove
in \ref{realization of afuglz}(1) that $\Vq:=\Vz\ot\mbq$ is isomorphic to $\afuglq$, which is the affine analogue of \cite[5.7]{BLM}. Furthermore, we will prove in \ref{realization of afuglz}(2) that $\afuglz$ is a $\mbz$-subalgebra of $\afuglq$ and $\Vz$ is the realization of $\afuglz$.
As a result, we derive an explicit $\mbz$-basis for $\afuglz$ together with explicit multiplication formulas between generators and arbitrary basis elements for $\afuglz$ (see \ref{formula in Vz},
\ref{basis-Vz} and \ref{realization of afuglz}). As a byproduct, we will  establish affine Schur-Weyl reciprocity at the integral level in \ref{affine Schur-Weyl duality}.


We organize this paper as follows. We recall the definition of Ringel--Hall algebras and extended Ringel--Hall algebras in \S2. Using Ringel--Hall algebas, we will construct a $\mbz$-submodule $\afuglz$ of $\afuglq$ in \S3. We review in \S4 the definition of affine quantum Schur algebras and generalize \cite[3.9]{BLM} to the affine case. Certain useful multiplication formulas for affine Schur algebras $\afSrmbz$ will be established in \S5. These formulas will be used to establish the stabilizaiton property for affine Schur algebras in \ref{stabilization property}. Then we use this property to construct an algebra $\sK$ without unity and derive some important multiplication formulas for the completion algebra $\hsKq$ of $\sK$ in \ref{formula in Vz}. In \ref{Vz-subalgebra}, we will use these formulas to prove that the $\mbz$-submodule $\Vz$ of $\hsKq$ constructed in \S8 is a $\mbz$-subalgebra of $\hsKq$. Finally, we will prove that
$\afuglq\cong\Vq$ and $\afuglz\cong\Vz$ in \ref{realization of afuglz}. Furthermore, we will prove that $\afuglz$ is a Hopf  algebra over $\mbz$ in \ref{Hopf Z-subalgebra}.  Using this realization for $\afuglz$, we will prove in \ref{affine Schur-Weyl duality} that the natural algebra homomorphism from $\sU_\mbz(\h{\frak{gl}}_n)$ to $\afSrmbz$ is surjective.

\begin{Not}\label{Notaion}\rm
For a positive integer $n$, let
$\afMnq$  be the set of all matrices
$A=(a_{i,j})_{i,j\in\mbz}$ with $a_{i,j}\in\mbq$ such that
\begin{itemize}
\item[(a)]$a_{i,j}=a_{i+n,j+n}$ for $i,j\in\mbz$; \item[(b)] for
every $i\in\mbz$, both sets $\{j\in\mbz\mid a_{i,j}\not=0\}$ and
$\{j\in\mbz\mid a_{j,i}\not=0\}$ are finite.
\end{itemize}
Let $\afMnz$ be the subset of $\afMnq$ consisting of matrices with integer entries.
Let
\begin{equation}\label{aftiThn,afThn}
\begin{split}
\aftiThn&:=\{A\in M_{\vtg,n}(\mbz)\mid a_{i,j}\geq 0,\,\forall i\not=j\},\
\afThn:=\{A\in\afMnz\mid a_{i,j}\in\mbn,\,\forall i,j
\}\\
\end{split}
\end{equation}


Let $\afmbzn=\{(\la_i)_{i\in\mbz}\mid
\la_i\in\mbz,\,\la_i=\la_{i-n}\ \text{for}\ i\in\mbz\}\text{ \,\,
and \,\,} \afmbnn=\{(\la_i)_{i\in\mbz}\in \afmbzn\mid \la_i\ge0\text{ for  }i\in\mbz\}.$
We will identify $\afmbzn$ with $\mbz^n$ via the following bijection
\begin{equation}\label{flat}
\flat:\afmbzn\lra\mbz^n,\quad \bfj\longmapsto \flat(\bfj)=(j_1,\cdots,j_n).
\end{equation}
There is a natural order relation $\leq$ on $\afmbzn$ defined by
\begin{equation}\label{order on afmbzn}
\la\leq\mu  \iff\la_i\leq \mu_i\text{ for all $1\leq i\leq n$}.
\end{equation}
We say that $\la<\mu$ if $\la\leq\mu$ and $\la\not=\mu$.


Let $\sZ=\mbz[\up,\up^{-1}]$, where $\up$ is an indeterminate, and let $\mbq(\up)$ be the fraction field
of $\sZ$.   Specializing $\up$ to $1$, $\mbq$ and
$\mbz$ will be viewed as $\sZ$-modules.
\end{Not}

\section{Ringel--Hall algebras and extended Ringel--Hall algebras}

Let $\tri$ ($n\geq 2$) be
the cyclic quiver\index{cyclic quiver}
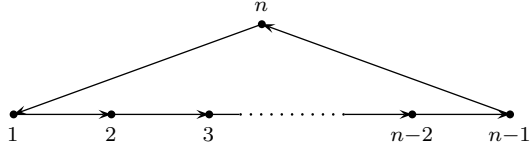
\begin{center}
\begin{pspicture}(-3,-.6)(3.6,1.6)
\psset{linewidth=0.5pt, arrowsize=3.5pt}
\psdot*(-3,0) \psdot*(-1.7,0) \psdot*(-.4,0) \psdot*(2.3,0)
\psdot*(3.6,0)\psdot*(.3,1.2) \uput[u](.3,1.2){$_{n}$}
\uput[d](-3,0){$_1$} \uput[d](-1.7,0){$_2$} \uput[d](-.4,0){$_3$}
\uput[d](2.3,0){$_{n-2}$} \uput[d](3.6,0){$_{n-1}$}
\psline{->}(-3,0)(-1.7,0) \psline{->}(-1.7,0)(-.4,0)
\psline(-.4,0)(0,0)
\psline[linestyle=dotted,linewidth=1pt](0,0)(1.4,0)\psline{->}(1.4,0)(2.3,0)
\psline{->}(2.3,0)(3.6,0) \psline{->}(3.6,0)(.3,1.2)
\psline{->}(.3,1.2)(-3,0)
\end{pspicture}
\end{center}
with vertex set $I=\mbz/n\mbz=\{1,2,\ldots,n\}$ and arrow set
$\{i\to i+1\mid i\in I\}$. Let $\field$ be a field.
A representation $V=(V_i,f_i)_{i\in I}$ of $\tri$
over $\field$ is called nilpotent if $f_n\cdots f_2f_1:V_1\ra V_1$ is nilpotent.
We will denote by $\rep^0\!\!\tri=\rep_\field^0\tri$ the category of finite-dimensional nilpotent representations of $\tri$ over $\field$.


For $i\in I$, let $S_i$
denote the one-dimensional representation in $\rep^0\!\!\tri$ with $(S_i)_i=\field$ and $(S_i)_j=0$ for $i\neq j$.
Let
\begin{equation*}
\begin{split}
\afThnp&=\{A\in\afThn\mid a_{i,j}=0\text{ for }i\geq j\}.
\end{split}
\end{equation*}
For any $A=(a_{i,j})\in\afThnp$, let
$$M(A)=M_\field(A)=\bop_{1\leq i\leq n,i<j}a_{i,j}M^{i,j},$$
where
$M^{i,j}$ is the unique indecomposable representation  of length $j-i$ with top $S_i$. For $A\in\afThnp$, let $\bfd(A)\in\mbn I=\mbn^n$ be the dimension vector of $M(A)$. We will sometimes identify $\mbn I$ with $\afmbnn$ under \eqref{flat}.


Given modules $M,N_1,\cdots,N_m$ in $\rep^0\!\!\tri$, let $F_{N_1\cdots
N_m}^M$ be the number of the filtrations
$$0=M_m\han M_{m-1}\han\cdots M_1\han M_0=M$$
such that $M_{t-1}/M_t\cong N_t$ for all $1\leq t\leq m$. By
\cite{Ri93} and \cite{Guo95}, for $A,B_1,\cdots,B_m\in\afThnp$,
there is a polynomial $\vi^{A}_{B_1\cdots B_m}\in\mbz[\up^2]$ in $\up^2$ such
that, for any finite field $\field$ of $q$ elements,
$\vi^{A}_{B_1\cdots B_m}|_{\up^2=q}=F_{M_{\field}(B_1)\cdots
M_{\field}(B_m)}^{M_{\field}(A)}.$
Moreover, for each $A\in \afThnp$, there is a polynomial
$\fka_A=\fka_A(\up^2)\in\mbz[\up^2]$ in $\up^2$ such that, for each finite
field $\field$ with $q$ elements,
$\fka_A|_{\up^2=q}=|\Aut(M_{\field}(A))|$.

For $\bfa=(a_i)\in\afmbzn$ and $\bfb=(b_i)\in\afmbzn$, the  Euler form  associated with the cyclic quiver $\tri$ is the bilinear form $\lan-,-\ran:\afmbzn\times\afmbzn\ra\mbz$ defined by
\begin{equation*}
\lan \bfa,\bfb\ran=\sum_{i\in I}a_ib_i-\sum_{i\in
I}a_ib_{i+1}.
\end{equation*}

Let $\Hall$ be the {\it (generic) Ringel--Hall algebra}
of the cyclic quiver
$\tri$, which is by definition the free module over
$\sZ=\mbz[\up,\up^{-1}]$ with basis $\{u_A=u_{[M(A)]}\mid A\in\afThnp\}$. The
multiplication is given by
$$u_{A}u_{B}=\up^{\lan \bfd(A),\bfd(B)\ran}\sum_{C\in\afThnp}\vi^{C}_{A,B}(\up^2)u_{C}$$
for $A,B\in \afThnp$.  For $A\in\afThnp$, let
\begin{equation*}\label{tilde u_A}
\ti u_A=v^{\dim \End(M(A))-\dim M(A)}u_A.
\end{equation*}


Now let us recall the triangular relation given in \cite[(9.2)]{DDX} for the Ringel--Hall algebra $\Hall$.
For $M,N\in\rep^0\!\!\tri$, there exists a unique extension $G$ (up to isomorphism)
of $M$ by $N$ with minimal $\dim\End(G)$, which will be denoted by $M*N$ in the sequel  (see \cite{DD05,Re01}).
Let $\sM$ be the set of isoclasses of nilpotent representations of
$\tri$ and define a multiplication $*$  on $\sM$ by $[M]*[N]=[M*N]$
for any $[M],[N]\in\sM$. Then by \cite[\S3]{DD05} $\sM$ is a monoid
with identity $1=[0]$.

For $\la\in\afmbnn$ let $$S_\la=\op_{i=1}^n\la_iS_i$$ be the semisimple representation in $\rep^0\!\!\tri$. A semisimple representation
$S_{\la}$ is called {\it sincere} if $\la$ is sincere, namely, all $\la_i$ are positive. For $1\leq i\leq n$ let  $\afbse_i\in\afmbnn$ be the element
satisfying $$\flat(\afbse_i)=\bse_i=(0,\cdots,0,\underset
{(i)}1,0,\cdots,0),$$ where $\flat$ is defined in \eqref{flat}.
Let
$$\widetilde I=\{\afbse_1,\afbse_2,\cdots,\afbse_n\}\cup\{\text{all sincere vectors in $\afmbnn$}\}.$$
Let $\ti\Sg$ be the set of words on the set $\ti I$. For
$w=\bsa_1\bsa_2\cdots \bsa_m\in\ti\Sg$, let $\wp^+(w)\in\afThnp$ be the element
defined by
$$[S_{\bsa_1}]*\cdots*[S_{\bsa_m}]=[M(\wp^+(w))].$$
Thus we obtain a map $\wp^+:\ti\Sg\ra\afThnp$.
Let $$\afThnm=\{A\in\afThn\mid a_{i,j}=0\text{ for }i\leq j\}.$$
For
$w=\bsa_1\bsa_2\cdots\bsa_m\in\ti\Sg$ we let
$$\tw=\bsa_m\cdots\bsa_2\bsa_1.$$
Let $\wp^-$ be the map from $\ti\Sg$
to $\afThnm$ defined by $\wp^-(w)=\! ^t(\wp^+(\tw))$,
where ${}^t(\wp^+(\tw))$ is the transpose of $\wp^+(\tw)$.  By
\cite[3.3]{DDX} the maps $\wp^+$ and $\wp^-$ are all surjective.

Following \cite[3.5]{BLM} and \cite{DF09} we may define the order relation $\pr$
on $\afMnz$ as follows.
For $A\in\afMnz$ and $i\not=j\in\mbz$, let
$$\sg_{i,j}(A)=\sum\limits_{s\leq i,t\geq j}a_{s,t}\text{ if $i<j$,}\text{ and }
\sg_{i,j}(A)=\sum\limits_{s\geq i,t\leq j}a_{s,t}  \text{ if
$i>j$}.$$ For $A,B\in\afMnz$, define
$B\pr A$ if $\sg_{i,j}(B)\leq\sg_{i,j}(A)$ for all $i\not=j$.
Put $B\p A$ if $B\pr A$ and, for some pair $(i,j)$ with $i\not=j$,
$\sg_{i,j}(B)<\sg_{i,j}(A)$.

For $\la\in\afmbnn$ let $u_\la=u_{[S_\la]}$ and let
$$\ti u_\la=\up^{\dim\End(S_\la)-\dim S_\la} u_{\la}.$$
Any word $w=\bsa_1\bsa_2\cdots\bsa_m$ in $\ti\Sg$ can be uniquely expressed in the {\it tight form} $w=\bsb_1^{x_1}\bsb_2^{x_2}\cdots\bsb_t^{x_t}$ where $x_i=1$ if $\bsb_i$ is sincere, and $x_i$ is the number of consecutive occurrences of $\bsb_i$ if $\bsb_i\in\{\afbse_1,\afbse_2,\cdots,\afbse_n\}$.
For $w=\bsa_1\bsa_2\cdots\bsa_m\in\ti\Sg$ with the tight form
$\bsb_1^{x_1}\bsb_2^{x_2}\cdots\bsb_t^{x_t}$  let
\begin{equation*}
\ti u_{(w)} =\ti u_{x_1\bsb_1}\ti u_{x_2\bsb_2}\cdots\ti u_{x_t\bsb_t}\in\Hall.
\end{equation*}
By \cite[(9.2)]{DDX} and \cite[6.2]{DF09}, we have the following triangular relation in $\Hall$.
\begin{Prop}\label{tri Hall}
For $A\in\afThnp$, there exist $w_A\in\ti\Sg$ such that $\wp^+(w_A)=A$ and
\begin{equation*}
\ti u_{(w_A)}=\sum_{B\in\afThnp\atop B\pr A,\,\bfd(A)=\bfd(B)}f_{{B,A}}\ti u_B
\end{equation*}
where $f_{{B,A}}\in\sZ$ and $f_{{A,A}}=1$. In particular, $\Hall$ is generated by $u_\la$ for $\la\in\afmbnn$.
\end{Prop}

Let $\bfHall=\Hall\ot\mbq(\up)$. The algebra $\bfHall$ does not have a Hopf algebra structure. However, if we add the torus algebra to $\bfHall$, we may get a Hopf algebra $\Hallpi$, called the extended Ringel--Hall algebra.
Let $\Hallpi$ be a $\mbq(\up)$-space with basis $\{u_A^+K_{\al}\mid \al\in\mbz I,A\in
\afThnp\}$. Let $\afThnp_1:=\afThnp\backslash\{0\}$.

\begin{Prop}\label{Green Xiao}
The $\mbq(\up)$-space $\Hallpi$ with basis $\{u_A^+K_{\al}\mid
\al\in\mbz I,A\in \afThnp\}$ becomes a Hopf algebra with the
following algebra, coalgebra and antipode structures.
\begin{itemize}
\item[(a)] Multiplication and unit:
\begin{equation*}
\begin{split}
u_A^+u_B^+&=\sum_{C\in\afThnp}\nup^{\lan \bfd(A),\bfd(B)\ran}\vi_{A,B}^C u_C^+, \text{\quad for all $A,B\in\afThnp$},\\
K_\al u_A^+&=\nup^{\lr{\bfd(A),\al}}u_A^+K_\al, \text{\quad for all $\al\in\mbz I$, $A\in\afThnp$},\\
K_{\al} K_{\bt}&=K_{\al+\bt}, \text{\quad for all $\al,\bt\in\mbz
I$}.
\end{split}
\end{equation*}
with unit $1=u_0^+=K_0$.

\item[(b)] Comultiplication and counit {\rm(Green \cite{Gr95})}:
\begin{equation*}
\begin{split}
\Dt(u_C^+)&=\sum_{A,B\in\afThnp}\nup^{\lan
\bfd(A),\bfd(B)\ran}\frac{\fka_A\fka_B}{\fka_C}
\vi_{A,B}^C u_B^+\ot u_A^+\ti K_{\bfd(B)},\\
\Dt(K_\al)&=K_\al\ot K_\al,\;\; \text{where $C\in\afThnp$ and
$\al\in\mbz I$},
\end{split}
\end{equation*}
with counit $\ep$ satisfying $\ep(u_C^+)=0$ for all
$C\in\afThnp_1$ and $\ep(K_\al)=1$ for all
$\al\in\mbz I$. Here, for each $\al=(a_i)\in\mbz I$,
$\ti K_\al$ denotes $(\ti K_1)^{a_1}\cdots (\ti K_n)^{a_n}$ with
$\ti K_i=K_iK_{i+1}^{-1}$.

\item[(c)] Antipode {\rm(Xiao \cite{X97})}:
\begin{equation*}
\begin{split}
&\sg(u_C^+)=\dt_{C,0}+\sum_{m\geq 1}(-1)^m\sum_{D\in\afThnp\atop
C_1,\ldots,C_m\in\afThnp_1}\frac{\fka_{C_1}\cdots
\fka_{C_m}}{\fka_C}\vi_{C_1,\ldots,C_m}^C \vi_{C_m,\ldots,C_1}^D
u_D^+\ti K_{-\bfd(C)},
\end{split}
\end{equation*}
for all $C\in\afThnp$, and $\sg(K_{\al})=K_{-\al}$, for all $\al\in\mbz
I$.
\end{itemize}
\end{Prop}

We conclude this section by introducing the integral form $\Hallp$ for $\Hallpi$.
For  $c,t\in\mathbb Z$ with $t\geq 1$, let
\begin{equation*}
\leb{K_i;c\atop t}\rib=\prod_{s=1}^{t}\frac{K_iv^{c-s+1}-K_i^{-1}
v^{-c+s-1}}{v^s-v^{-s}}\,\,\text{ and }\,\,\leb{K_i;c\atop 0}\rib=1.
\end{equation*}
Let $\Hallp$ be the $\sZ$-submodule of $\Hallpi$ spanned by all  $u_A^+\prod_{1\leq i\leq n}\leb{K_i;0\atop \la_i}\rib K_i^{\dt_i}$, where $A\in\afThnp$, $\dt_i\in\{0,1\}$ and $\la\in\afmbnn$.

\begin{Lem}\label{Hallp}
$\Hallp$ is a $\sZ$-Hopf subalgebra of $\Hallpi$.
\end{Lem}
\begin{proof}
Clearly, for $A\in\afThnp$, $1\leq i\leq n$ and $t\in\mbn$,
\begin{equation}\label{eq 1-Hallp}
\leb{K_i;0\atop t}\rib u_A^+=u_A^+\leb{K_i;\lr{\bfd(A),\afbse_i}\atop t}\rib.
\end{equation}
This together with \cite[2.14]{Lu90} implies that
$\Hallp$ is a $\sZ$-subalgebra of $\Hallpi$.

For $m\in\mbn$, let $[\![m]\!]^!=[\![1]\!][\![2]\!]\cdots[\![m]\!]$ where $[\![i]\!]=\frac{\up^i-\up^{-i}}{\up-\up^{-1}}$. Clearly, for  $\la,\la^{(1)},\cdots,\la^{(m)}\in\afmbnn$ with $\la=\sum_{1\leq i\leq m}\la^{(i)}$,
$$
\vi^{\la}_{\la^{(1)}\cdots\la^{(m)}}
=\prod_{1\leq i\leq m}\dleb{\la_i\atop\la^{(1)}_i\cdots\la^{(m)}_i}
\drib,\qquad \fka_\la=\prod_{1\leq i\leq n\atop 0\leq s\leq \la_i-1}(\up^{2\la_i}-\up^{2s}),
$$
where $\dleb{\la_i\atop\la^{(1)}_i,\ldots,
\la_i^{(m)}}\drib=\frac{[\![\la_i]\!]^!}{[\![\la^{(1)}_i]\!]^!\ldots
[\![\la_i^{(m)}]\!]^!}$.
We conclude that
\begin{equation}\label{eq 3-Hallp}
\Dt(\ti u_\la^+)=
\sum_{\la=\la^{(1)}+\la^{(2)}\atop\la^{(i)}\in\afmbnn,\,\forall i}\up^{\lr{\la^{(2)},\la^{(1)}}}\ti u_{\la^{(1)}}^+\ot\ti u_{\la^{(2)}}^+\ti K_{\la^{(1)}}\in\Hallp\ot\Hallp
\end{equation}
and
\begin{equation}\label{eq 4-Hallp}
\begin{split}
\sg(u_\la^+)&=\dt_{\la,0}+\sum_{m\geq 1}(-1)^m
\sum_{D\in\afThnp\atop\la^{(i)}\in\afmbnn,\,\la^{(i)}\not=0,\,
1\leq i\leq m}\up^{-2\sum_{1\leq i\leq n}\sum_{1\leq k<l\leq m}\la^{(k)}_i\la_i^{(l)}}\vi^{D}_{\la^{(m)}\cdots\la^{(1)}}
u_{D}^+\ti K_{-\la}\\
&\in\Hallp.
\end{split}
\end{equation}
Furthermore, by \cite[(2.1),(2.2)]{T}, we have
\begin{equation}\label{eq 2-Hallp}
\Dt\left(\leb{K_i;0\atop t}\rib\right)
=\sum_{0\leq j\leq t}K_i^{-t+j}\leb{K_i;0\atop j}\rib
\ot K_i^{j}\leb{K_i;0\atop t-j}\rib\in\Hallp\ot\Hallp
\end{equation}
and
\begin{equation}
\sg\left(\leb{K_i;0\atop t}\rib\right)=(-1)^t\leb{K_i;-1+t\atop t}\rib\in\Hallp.
\end{equation}
Consequently, $\Hallp$ is a Hopf algebra over $\sZ$, since $\Hallp$ is generated by $u_\la^+$, $K_i^{\pm 1}$ and $\leb{K_i;0\atop t}\rib$ for all $\la\in\afmbnn$, $1\leq i\leq n$ and $t\in\mbn$ by \ref{tri Hall}.
\end{proof}


\section{The integral form $\afuglz$ of $\afuglq$}
Let ${\frak{gl}_n}(\mbq)$ be the general linear Lie algebra over $\mbq$. Let  $$\afgl:={\frak{gl}_n}(\mbq)\ot\mbq[t,t^{-1}]$$ be the loop algebra associated to $\frak{gl}_n(\mbq)$.
For $i,j\in\mbz$, let $\afE_{i,j}\in\afMnq$ be the matrix
$(e^{i,j}_{k,l})_{k,l\in\mbz}$ defined by
\begin{equation*}e_{k,l}^{i,j}=
\begin{cases}1&\text{if $k=i+sn,l=j+sn$ for some $s\in\mbz$;}\\
0&\text{otherwise},\end{cases}
\end{equation*}
Recall the set $\afMnq$ defined in \ref{Notaion}.
Clearly, the map
$$\afMnq\lra\afgl ,\,\,\,\afE_{i,j+ln}\longmapsto E_{i,j}\ot t^l, \,1\le i,j\le n,l\in\mbz $$
is a Lie algebra isomorphism. We will identify the loop algebra
$\afgl$ with $\afMnq$ in the sequel.

Let $\afuglq$ be the universal enveloping algebra of the loop
algebra $\afgl$. Let $\afuglqp$ (resp., $\afuglqm$) be the subalgebra of $\afuglq$ generated
by $\afE_{i,j}$ (resp., $\afE_{j,i}$) for all $i<j$.
Let $\afuglqz$ be the subalgebra of $\afuglq$ generated
by $\afE_{i,i}$ for $1\leq i\leq n$.

We may interpret the $\pm$-part of $\afuglq$ as the specialization of Hall algebras.
Let $\Hall_\mbq=\Hall\ot_\sZ\mbq$, where $\mbq$ is regarded as a $\sZ$-module by specializing $\up$ to $1$. We shall denote $u_{A}\otimes 1$ by $u_{A,1}$ for $A\in\afThnp$.
\begin{Lem}[{\cite[6.1.2]{DDF}}]\label{iotap}
$(1)$ The set
$$\bigg\{\prod_{1\leq i\leq n\atop i<j,\,j\in\mbz}u_{_{\afE_{i,j},1}}^{a_{i,j}}\,\bigg|\,
A=(a_{i,j})\in\afThnp\bigg\}$$
forms a $\mbq$-basis of $\Hall_\mbq$, where the products are taken with respect to any fixed total order on
$\{(i,j)\mid 1\leq i\leq n,\,i<j,\,j\in\mbz\}$.

$(2)$ There is a unique injective algebra homomorphism
$\iota^+:\Hall_\mbq\ra\afuglq$ (resp., $\iota^-:\Hall_\mbq^{\mathrm op}\ra\afuglq$) taking $u_{{\afE_{i,j},1}}\mapsto\afE_{i,j}$ (resp.,$u_{{\afE_{i,j},1}}\mapsto\afE_{j,i}$) for all $i<j$ such that $\iota^+(\Hall_\mbq)=\afuglqp$ and $\iota^-(\Hall_\mbq^{\mathrm op})=\afuglqm$.
\end{Lem}

Let $\afuglqpz=\afuglqp\afuglqz$ be the the Borel subalgebra $\afuglqpz$ of $\afuglq$. We now show that the algebra isomorphism $\iota^+:\Hall_\mbq\ra\afuglqp$ can be extended to a Hopf algebra isomorphism between the specialization $\ol{\Hallp_\mbq}$ of $\Hallp$ at $\up=1,\,K_i=1$ and $\afuglqpz$.

Let $\Hallp_{\mbq}=\Hallp\ot_\sZ\mbq$.
We shall denote $u_{A,1}^+=u_{A}^+\otimes 1$, $K_{i,1}=K_i\ot 1$ and $\leb{K_i;0\atop t}\rib_1=\leb{K_i;0\atop t}\rib\ot 1$ for $A\in\afThnp$, $1\leq i\leq n$ and $t\in\mbn$.
Let $$\ol{\Hallp_\mbq}=\Hallp_{\mbq}/\lr{K_{i,1}-1\mid 1\leq i\leq n}.$$
We will use the same notation for elements in
$\Hallp_{\mbq}$ and $\ol{\Hallp_\mbq}$. The algebra $\ol{\Hallp_\mbq}$ inherits a Hopf algebra structure from that of $\Hallp$.

\begin{Lem}\label{Hopf algebra isomorphism}
There is a  Hopf algebra isomorphism $$\iota^{\geq 0}:\ol{\Hallp_\mbq}\ra\afuglqpz$$
defined by sending $u_{{\afE_{i,j},1}}^+$ to $\afE_{i,j}$, and
$\leb{K_i;0\atop 1}\rib_1$ to $\afE_{i,i}$ for $i<j$.
\end{Lem}
\begin{proof}
By \cite[2.3(g9),(g10)]{Lu90} we have in $\ol{\Hallp_\mbq}$,
$\leb{K_i;\pm 1\atop 1}\rib_1=\leb{K_i;0\atop 1}\rib_1\pm 1.$
Thus by \eqref{eq 1-Hallp} we have in $\ol{\Hallp_\mbq}$,
\begin{equation*}
\begin{split}
\leb{K_i;0\atop 1}\rib_1 u_{\afE_{k,l},1}^+
-u_{\afE_{k,l},1}^+\leb{K_i;0\atop 1}\rib_1 =u_{\afE_{k,l},1}^+\left(
\leb{K_i;\dt_{\bar i,\bar k}-\dt_{\bar i,\bar l}\atop 1}\rib_1-\leb{K_i;0\atop 1}\rib_1\right)=\left( \dt_{\bar i,\bar k}-\dt_{\bar i,\bar l} \right)u_{\afE_{k,l},1}^+.
\end{split}
\end{equation*}
for $1\leq i\leq n$ and $k<l$.
This together with \ref{iotap}(2) implies that there is an algebra homomorphism
$f:\afuglqpz\ra\ol{\Hallp_\mbq}$ such that $f(\afE_{i,j})=u_{{\afE_{i,j},1}}^+$ and $f(\afE_{i,i})=\leb{K_i;0\atop 1}\rib_1$ for $i<j$.
Since, by \cite[4.1(d)]{Lu89},
$\leb {K_i;0\atop 1}\rib_1-\leb{K_i;-j\atop 1}\rib_1=j$ for  $1\leq i\leq n$ and $j\in\mbz$, it follows from \cite[4.1(f)]{Lu89}
that, we have in  $\ol{\Hallp_\mbq}$,
\begin{equation}\label{eq Hopf algebra isomorphism}
\leb{K_i;0\atop t}\rib_1=\frac{1}{r!}\prod_{0\leq j\leq t-1}\leb{K_i;-j\atop 1}\rib_1=f\left(\bpa{\afE_{i,i}}{t}\right),
\end{equation}
where $$\bigg({\afE_{i,i}\atop
t}\bigg)=\frac{\afE_{i,i}(\afE_{i,i}-1)\cdots(\afE_{i,i}-t+1)}{t!}.$$
Thus by \cite[2.14]{Lu90} and \ref{iotap}(1) $f$ sends the PBW-basis
$$\bigg\{\prod_{1\leq i\leq n\atop i<j,\,j\in\mbz}
(\afE_{i,j})^{a_{i,j}}\prod_{1\leq i\leq n}\bpa{\afE_{i,i}}{t}\,\bigg|\,
A=(a_{i,j})\in\afThnp,\,\la\in\afmbnn\bigg\}$$
for $\afuglqpz$ to a $\mbq$-basis for $\ol{\Hallp_\mbq}$. Consequently, $f$ is an algebra isomorphism. Finally, since $$\Dt(u^+_{\afE_{i,j},1})=u^+_{\afE_{i,j},1}\ot 1+1\ot u^+_{\afE_{i,j},1},\qquad\Dt\left(\leb{K_i;0\atop 1}\rib_1\right)=\leb{K_i;0\atop 1}\rib_1\ot 1+1\ot\leb{K_i;0\atop 1}\rib_1$$ and $$\sg(u^+_{\afE_{i,j},1})=-u^+_{\afE_{i,j},1},\quad \sg\left(\leb{K_i;0\atop 1}\rib_1\right)=-\leb{K_i;0\atop 1}\rib_1$$ for $i<j$, we conclude that $f$ is a Hopf algebra isomorphism.
\end{proof}

We now use Ringel--Hall algebras to define the integral form $\afuglz$ of $\afuglq$.
Let $\afuglzp=\iota^+(\Hall_\mbz)$ and
$\afuglzm=\iota^-(\Hall_\mbz^{\mathrm{op}})$, where
$\Hall_\mbz=\Hall\ot_\sZ\mbz$ and $\mbz$ is regarded as a $\sZ$-module by specializing $\up$ to $1$. Let $\afuglzz$ be the $\mbz$-submodule of $\afuglq$
spanned by $\prod_{1\leq i\leq n}\bpa{\afE_{i,i}}{\la_i}$, for  $\la\in\afmbnn$.
Let
\begin{equation*}
\begin{split}
\afuglz&=\afuglzp
\afuglzz\afuglzm.\\
\end{split}
\end{equation*}
We will prove that $\afuglz$ is a $\mbz$-subalgebra of $\afuglq$ and give a BLM realization of $\afuglz$ in \ref{realization of afuglz}. Furthermore, we will use \ref{Hopf algebra isomorphism} to show that $\afuglz$ is a $\mbz$-Hopf subalgebra of $\afuglq$ in \ref{Hopf Z-subalgebra}.


\section{Affine quantum Schur algebras}
Let $\afsygr$ be the group consisting of all permutations
$w:\mbz\ra\mbz$ such that $w(i+r)=w(i)+r$ for $i\in\mbz$.
Let $W$ be the subgroup of
$\afsygr$ consisting of $w\in\afsygr$ with
$\sum_{i=1}^rw(i)=\sum_{i=1}^ri$. By \cite{Lu83}, $W$
is the Weyl group of affine type $A$ with generators $s_i$ ($1\leq i\leq r$) defined by
setting
$s_i(j)=j$ for $j\not\equiv i,i+1\mnmod r$, $s_i(j)=j-1$ for
$j\equiv i+1\mnmod r$ and $s_i(j)=j+1$ for $j\equiv i\mnmod r$.
The subgroup of $\affSr$ generated by $s_1,\ldots,s_{r-1}$ is isomorphic to the symmetric group $\fS_r$.
Let $\rho$ be the permutation of $\mbz$ sending $j$ to $j+1$ for all $j\in\mbz$. Then $\afsygr=\langle\rho\rangle\ltimes W.$
We extend the length function $\ell$ on $W$ to $\afsygr$ by setting $\ell(\rho^mw)=\ell(w)$ for all $m\in\mbz,w\in W$.

The extended affine Hecke algebra $\afHr$ over $\sZ$ associated to
$\affSr$ is the (unital) $\sZ$-algebra with basis
$\{T_w\}_{w\in\affSr}$, and multiplication defined by
\begin{equation*}
\begin{cases} T_{s_i}^2=(\up^2-1)T_{s_i}+\up^2,\quad&\text{for }1\leq i\leq n\\
T_{w}T_{w'}=T_{ww'},\quad&\text{if}\
\ell(ww')=\ell(w)+\ell(w').
\end{cases}
\end{equation*}

For $\la=(\la_i)_{i\in\mbz}\in\afmbzn$
let
$
\sg(\la)=\sum_{1\leq i\leq n}\la_i$.
For $r\geq 0$  we set
\begin{equation*}
\afLanr=\{\la\in\afmbnn\mid\sg(\la)=r\}.
\end{equation*}
For $\la\in\afLanr$, let $\fS_\la:=\fS_{(\la_1,\ldots,\la_n)}$
be the corresponding standard Young subgroup of $\fS_r$.
For each $\la\in\afLanr$, let
$x_\la=\sum_{w\in\fS_\la}T_w\in\afHr$. The endomorphism algebras
$$\sS_\vtg(n,r):=\End_{\afHr}\biggl
(\bop_{\la\in\La_\vtg(n,r)}x_\la\afHr\biggr).$$
are called affine quantum Schur algebras (cf. \cite{GV,Gr99,Lu99}).

Following \cite{Gr99}, we will introduce a $\sZ$-basis of $\afSr$ as follows.
For $\la\in\afLanr$, let
$$\afmsD_\la=\{d\mid d\in\affSr,\ell(wd)=\ell(w)+\ell(d)\text{ for
$w\in\fS_\la$}\}.$$
Note that we have
\begin{equation}\label{minimal coset representative}
\aligned
d^{-1}\in\afmsD_\la
&\iff d(\la_{0,i-1}+1)<d(\la_{0,i-1}+2)<\cdots<d(\la_{0,i-1}+\la_i),\,\forall 1\leq i\leq n,\endaligned
\end{equation}
where $\la_{0,i-1}=\sum_{1\leq t\leq i-1}\la_t$.
Let $\afmsD_{\la,\mu}=\afmsD_{\la}\cap{\afmsD_{\mu}}^{-1}$.
For $\la,\mu\in\afLanr$ and $d\in\afmsD_{\la,\mu}$, define
$\phi_{\la,\mu}^d\in\sS_\vtg(n,r)$ as follows:
\begin{equation}\label{def of standard basis}
\phi_{\la,\mu}^d(x_\nu h)=\dt_{\mu\nu}\sum_{w\in\fS_\la
d\fS_\mu}T_wh
\end{equation}
where $\nu\in\afLanr$ and $h\in\afHr$. Then by \cite{Gr99} the set
$\{\phi_{\la,\mu}^d\mid \la,\mu\in\afLanr,\,
d\in\afmsD_{\la,\mu}\}$ forms a basis for $\sS_\vtg(n,r)$.

Recall the sets $\aftiThn$ and $\afThn$ defined in \eqref{aftiThn,afThn}.
For $A\in\aftiThn$,
let
$
\sg(A)=\sum_{1\leq i\leq n,\,
j\in\mbz}a_{i,j}.$
For $r\geq 0$, let
\begin{equation*}
\afThnr=\{A\in\afThn\mid\sg(A)=r\}.
\end{equation*}
The basis for $\afSr$ is indexed by triples
$\{(\la,d,\mu)\mid \la,\mu\in\afLanr,\,
d\in\afmsD_{\la,\mu}\}$. In what follows next, it will be
convenient to reindex these basis elements by the set $\afThnr$. For $1\leq i\leq n$, $k\in\mbz$ and $\la\in\afLa(n,r)$ let
\begin{equation*}
R_{i+kn}^{\la}=\{\la_{k,i-1}+1,\la_{k,i-1}+2,\ldots,\la_{k,i-1}+\la_i=\la_{k,i}\},
\end{equation*}
where $\la_{k,i-1}=kr+\sum_{1\leq t\leq i-1}\la_t$.
By \cite[7.4]{VV99} (see also \cite[9.2]{DF09}), there is
a bijective map
\begin{equation*}
{\jmath_\vtg}:\{(\la, d,\mu)\mid
d\in\afmsD_{\la,\mu},\la,\mu\in\afLanr\}\lra\afThnr
 \end{equation*}
sending $(\la, w,\mu)$ to $A=(a_{k,l})$, where $a_{k,l}=|R_k^\la\cap wR_l^\mu|$ for all $k,l\in\mbz$.
If $\la,\mu\in\afLanr$ and $d\in\afmsD_{\la,\mu}$ are such that $A=\jmath_\vtg(\la,d,\mu)$, we will denote $\phi_{\la,\mu}^d$ by $e_A$ and let
\begin{equation*}
[A]=\up^{-d_A}e_{A},\quad\text{ where } \quad
d_{A}=\sum_{1\leq i\leq n\atop i\geq k,j<l}a_{i,j}a_{k,l}.
\end{equation*}
By \cite[1.11]{Lu99}, the $\sZ$-linear map
\begin{equation}\label{taur}
\tau_r:\afSr\lra\afSr,\;\;[A]\lm [\tA]
\end{equation} is an algebra
anti-involution, where $\tA$ is the transpose of $A$.

Let
$$\afThnpm=\{A\in\afThn\mid a_{i,i}
=0\text{ for all $i$}\}.$$
For $A\in\afThnpm$ and $\bfj\in\afmbzn$, define $A(\bfj,r)\in
\afSr$ by
\begin{equation*}
A(\bfj,r)=
\sum_{\la\in\La_\vtg(n,r-\sg(A))}\up^{\sum_{1\leq i\leq n}\la_ij_i}[A+\diag(\la)].
\end{equation*}
The affine quantum Schur algebra $\afSr$ and the Ring-Hall
algebra $\Hall$ can be related by the following
algebra homomorphism defined in \cite[7.6]{VV99}.
\begin{Prop} \label{afzrpm}
{\rm(1)} There is a $\sZ$-algebra homomorphism
$$\etarm:\Hall^{\rm op}\lra\afSr,\;\;\ti u_A\lm (\tA)(\bfl,r)\;\;
\text{for all $A\in\afThnp$}.$$

{\rm(2)} Dually, there is a $\sZ$-algebra homomorphism
$$\etarp:\Hall\lra\afSr,\;\;\ti u_A \lm A(\bfl,r)\;\;\text{for all
$A\in\afThnp$}.$$
\end{Prop}

We end this section by generalizing \cite[3.9]{BLM} to the affine case. This is the first key result in proving stabilization property of
multiplication for affine Schur algebras.

First we will use the triangular relation  for Ringel--Hall algebras to get similar relations for affine quantum Schur algebras.
For $\bsa\in\afmbnn$ let
\begin{equation}\label{Aa}
A_\bsa=\sum_{1\leq i\leq n}a_i\afE_{i,i+1}\in\afThnp,\quad
B_\bsa={}^t (A_\bsa)=\sum_{1\leq i\leq n}a_i\afE_{i+1,i}\in\afThnm.
\end{equation}
For $w=\bsa_1\bsa_2\cdots\bsa_m\in\ti\Sg$ with the tight form
$\bsb_1^{x_1}\bsb_2^{x_2}\cdots\bsb_t^{x_t}$, let
\begin{equation*}
\begin{split}
\ttm_{(w),r}^+&=\etarp(\ti u_{(w)})=A_{x_1\bsb_1}(\bfl,r)A_{x_2\bsb_2}(\bfl,r)\cdots
A_{x_t\bsb_t}(\bfl,r)\in\afSr,\\
\ttm_{(w),r}^-&=B_{x_1\bsb_1}(\bfl,r)B_{x_2\bsb_2}(\bfl,r)\cdots
B_{x_t\bsb_t}(\bfl,r)\in\afSr.
\end{split}
\end{equation*}
For $A\in\aftiThn$, we write
\begin{equation}\label{A^+,A^-,A^0}
A=A^++A^0+A^-
\end{equation}
 where $A^+\in\afThnp$,
$A^-\in\afThnm$ and $A^0$ is a diagonal matrix.

\begin{Lem}
For any $A\in\afThnpm$, there exist $w_{A^+},w_{A^-}\in\ti\Sg$ such that $\wp^+(w_{A^+})=A^+$, $\wp^-(w_{A^-})=A^-$ and
\begin{equation}\label{tri-positive-negative}
\begin{split}
\ttm_{(w_{A^+}),r}^+&=\sum_{B\in\afThnp,\,B\pr A^+\atop\bfd(A^+)=\bfd(B)}f_{{B,A^+}}B(\bfl,r)\qquad
\ttm_{(w_{A^-}),r}^-=\sum_{B\in\afThnm,\,B\pr A^-\atop \bfd(\tB)=\bfd({}^t\!(A^-))}g_{{B,A^-}}B(\bfl,r)
\end{split}
\end{equation}
for any $r\geq 0$, where $f_{{B,A^+}},g_{{B,A^-}}\in\sZ$ is independent of $r$ and $f_{{A^+,A^+}}=g_{{A^-,A^-}}=1$.
\end{Lem}
\begin{proof}
The first equation follows from \ref{tri Hall} and \ref{afzrpm}.
Now we assume $A\in\afThnm$. Then $\tA\in\afThnp$. Thus there exist $w \in\ti\Sg$ such that $\wp^+(w)=\tA$ and
\begin{equation*}
\ttm_{(w),r}^+=\sum_{B\in\afThnp,\,B\pr \tA\atop\bfd(\tA)=\bfd(B)}f_{{B,\tA}}B(\bfl,r).
\end{equation*}
Since $X\pr Y$ if and only if $\tX\pr\tY$ for
$X,Y\in\afThn$, applying the antiautomorphism $\tau_r$ defined in
\eqref{taur} to the above equation yields
\begin{equation*}
\ttm_{(w_A),r}^- =\tau_r(\ttm_{(w),r}^+)=\sum_{B\in\afThnp,\,B\pr \tA\atop\bfd(\tA)=\bfd(B)}f_{{B,\tA}}(\tB)(\bfl,r)
=\sum_{C\in\afThnm,\,C\pr A\atop\bfd(\tA)=\bfd(\tC)}g_{C,A}C(\bfl,r),
\end{equation*}
where $w_A=\tw$ and $g_{C,A}=f_{{\tC,\tA}}$.
\end{proof}

For $A\in\aftiThn$ let
\begin{equation*}
\begin{split}
\ro(A)=\bigl(\sum_{j\in\mbz}a_{i,j}\bigr)_{i\in\mbz},\,\,&\qquad
\co(A)=\bigl(\sum_{i\in\mbz}a_{i,j}\bigr)_{j\in\mbz}\in\afmbzn.\\
\end{split}\end{equation*}
For $A\in\aftiThn$ with $\sg(A)=r$, we denote $[A]=0\in\afSr$ if $a_{i,i}<0$ for some $i\in\mbz$.
For $A\in\afThnr$, let
\begin{equation*}
\bfsg(A)=(\sg_i(A))_{i\in\mbz}\in\afLanr
\end{equation*}
where
$\sg_{i}(A)=a_{i,i}+\sum_{j<i}(a_{i,j}+a_{j,i}).$
In the following discussion we shall denote by $[A]+\text{lower terms}$ an element of $\afSr$ which is equal to $[A]$ plus a $\sZ$-linear combination of elements $[A']$ with $A'\in\afThnr$, $A'\p A$, $\co(A')=\co(A)$ and $\ro(A')=\ro(A)$.
The following triangular relation for affine quantum Schur algebras is given in \cite[3.7.7]{DDF}.
\begin{Prop}\label{tri-affine Schur algebras}
For $A\in\afThnpm$ and $\la\in\afLanr$, we have
\begin{equation*}
A^+(\bfl,r)[\diag(\la)]A^-(\bfl,r)
=[A+\diag(\la-\bfsg(A))]+\text{lower terms}.
\end{equation*}
In particular, the set $$\{A^+(\bfl,r)[\diag(\la)]A^-(\bfl,r)\mid A\in\afThnpm,\,\la\in\afLanr,\,\la\geq\bfsg(A)\}$$ forms a $\sZ$-basis for $\afSr$, where the order relation $\leq$ is defined in \eqref{order on afmbzn}.
\end{Prop}

\begin{Lem}\label{lem for generalization of 3.9 in BLM}
Let $A_i,B_j\in\afThnpm$ ($1\leq i\leq s$, $1\leq j\leq t$) and $\la\in\afLanr$. Then we have
\begin{equation*}
\begin{split}
&\quad\quad A_1(\bfl,r)\cdots A_s(\bfl,r)[\diag(\la)]B_1(\bfl,r)\cdots B_t(\bfl,r)\\
&=
[A_1+\diag(\la^{(1)})]\cdots[A_s+\diag(\la^{(s)})]
[B_1+\diag(\mu^{(1)})]\cdots[B_t+\diag(\mu^{(t)})]
\end{split}
\end{equation*}
where $\la^{(i)}=\la-\co(A_i)+\sum_{i+1\leq k\leq s}(\ro(A_k)-\co(A_k))$ and $\mu^{(j)}= \la-\ro(B_j)+\sum_{1\leq k\leq j-1}(\co(B_k)-\ro(B_k))$ for $1\leq i\leq s$ and $1\leq j\leq t$.
\end{Lem}
\begin{proof}
Clearly,  for any $A\in\afThnpm$ and $\la\in\afLanr$, we have $[\diag(\la)]=[\diag(\la)]^2$ and
\begin{equation}\label{commutative formula}
A(\bfl,r)[\diag(\la)]=[\diag(\la-\co(A)+\ro(A))]A(\bfl,r).
\end{equation}
Repeatedly using \eqref{commutative formula}, we conclude
the assertion.
\end{proof}

Now we can prove the affine version of \cite[3.9]{BLM}.
\begin{Prop}\label{generalization of BLM 3.9}
Let $A\in\afTh(n,b)$. We choose $w_{A^+},w_{A^-}\in\ti\Sg$ such that \eqref{tri-positive-negative} hold. We assume $w_{A^+}$ and $w_{A^-}$ have tight form  $w_{A^+}=\bsa_1^{x_1}\cdots\bsa_s^{x_s}$, $w_{A^-}=\bsb_1^{y_1}\cdots\bsb_t^{y_t}$ and let $A_i=A_{x_i\bsa_i}$, $B_j=B_{y_j\bsb_j}$ for $1\leq i\leq s$ and $1\leq j\leq t$. Then the following identity holds in $\sS_\vtg(n,an+b)$
\begin{equation*}
\begin{split}
&[{}_a(A_1+\diag(\la^{(1)}))]
\cdots
[{}_a(A_s+\diag(\la^{(s)}))]
[{}_a(B_1+\diag(\mu^{(1)}))]
\cdots
[{}_a(B_t+\diag(\mu^{(t)}))]\\
&\qquad=[{}_aA]+\text{lower terms}
\end{split}
\end{equation*}
for any $a\geq 0$, where $\la^{(i)}=\bfsg(A)-\co(A_i)+\sum_{i+1\leq k\leq s}(\ro(A_k)-\co(A_k))$ and $\mu^{(j)}=\bfsg(A)-\ro(B_j)+\sum_{1\leq k\leq j-1}(\co(B_k)-\ro(B_k))$ for $1\leq i\leq s$ and $1\leq j\leq t$.
\end{Prop}
\begin{proof}
By  \eqref{tri-positive-negative} and \ref{tri-affine Schur algebras},
for any $r\geq 0$ and $\la\in\afLanr$, we have
$$\ttm_{(w_{A^+}),r}^+[\diag(\la)]\ttm_{(w_{A^-}),r}^-
=[A^\pm+\diag(\la-\bfsg(A^\pm))]+\text{lower terms}.$$
Now the assertion follows from \ref{lem for generalization of 3.9 in BLM}.
\end{proof}

\section{The fundamental multiplication formulas for affine Schur algebras}
We derive certain useful multiplication formulas for affine Schur algebras in \ref{funtmental multiplication formula}. These formulas is the second key result for the proof of the stabilization property of
multiplication for affine Schur algebras.

We need  some preparation before proving \ref{funtmental multiplication formula}.
For a finite subset $X\han\affSr$, let $\ul X=\sum_{x\in
X}x\in\mbq\affSr$.
It is clear that for
$\la,\mu\in\afLanr$ and $w\in\affSr$,
\begin{equation}\label{conjugate intersection}
\ul{{\frak S}_\la} w\ul{\frak S_\mu}=|w^{-1} \frak S_\la
w\cap\frak S_\mu|\ul{{\frak S}_\la w\frak S_\mu}.
\end{equation}

\begin{Lem}[{\cite[3.2.3]{DDF}}]
\label{double coset}
Let $\la,\mu\in\afLanr$ and $d\in\msD_{\la,\mu}^\vtg$. Assume
$A=\jmath_\vtg(\la,d,\mu)$. Then $d^{-1}\frak S_\la
d\cap\frak S_\mu=\frak S_\nu$, where
$\nu=(\nu^{(1)},\ldots,\nu^{(n)})$ and
$\nu^{(i)}=(a_{ki})_{k\in\mbz}=(\ldots,a_{1i},\ldots,a_{ni},\ldots)$.
\end{Lem}


\begin{Lem}\label{vartheta}
Assume $\mu\in\afLanr$, $\bt\in\afmbnn$ is such that
$\mu\geq\bt$.

$(1)$
Let $\al=\sum_{1\leq i\leq n}(\mu_i-\bt_i)\afbse_{i-1}$, $\dt=(\al_n,\bt_1,\al_1,\bt_2,\cdots,\al_{n-1},\bt_n)$ and
$$\msY=\{(Y_0,Y_1,\cdots,Y_{n-1})\mid Y_i\han R_{i+1}^\mu,\,|Y_i|=\al_i,\ for\ 0\leq i\leq n-1\}.$$ Then there is a bijective map $$\vartheta:\afmsD_\dt\cap\fS_\mu\ra\msY$$ defined by sending $w$ to $(w^{-1}X_0,w^{-1}X_1,\cdots,w^{-1}X_{n-1})$ where
$X_i=
\{\mu_{0,i}+1,\mu_{0,i}+2,\cdots,\mu_{0,i}+\al_i\},$ with $\mu_{0,i}=\sum_{1\leq s\leq i}\mu_s$ and $\mu_{0,0}=0$.

$(2)$ Let $\ga=\mu-\bt$,   $\th=(\bt_1,\ga_1,\bt_2,\ga_2,\cdots,\bt_n,\ga_n)$ and
$$\msY'=\{(Y_1',Y_2',\cdots,Y_{n}')\mid Y_i'\han R_{i}^{\mu},\,|Y_i'|=\ga_i,\ for\ 1\leq i\leq n\}.$$ Then there is a bijective map $$\vartheta':\afmsD_{\th}\cap\fS_{\mu}\ra\msY'$$ defined by sending $w$ to $(w^{-1}X_1',w^{-1}X_2',\cdots,w^{-1}X_{n}')$ where
$X_i'=
\{\mu_{0,i-1}+\beta_i+1,\mu_{0,i-1}+\beta_i+2,\cdots,\mu_{0,i}\}.$
\end{Lem}
\begin{proof}
We only prove (1). The proof of (2) is similar.


We assume $w_1,w_2\in\afmsD_\dt\cap\fS_\mu$ is such that $\vartheta(w_1)=\vartheta(w_2)$.  Since $w_1,w_2\in\afmsD_\dt$, by \eqref{minimal coset representative}, for $1\leq i\leq 2$ we have
\begin{equation}\label{inequality}
w_i^{-1}(j)<w_i^{-1}(j+1)
\end{equation}
if either $\{j,j+1\}\han X_k$ or $\{j,j+1\}\han R_{k+1}^\mu\backslash X_k$ for some $0\leq k\leq n-1$. Furthermore, we have  $w_1^{-1}(R_{i+1}^\mu\backslash X_i)=w_2^{-1}(R_{i+1}^\mu\backslash X_i)$ for $0\leq i\leq n-1$, since $w_1^{-1}X_i=w_2^{-1}X_i$ and $w_1,w_2\in\fS_\mu$.  Thus, $w_1^{-1}(j)=w_2^{-1}(j)$ for $j\in \bigcup_{0\leq k\leq n-1}(X_k\cup R_{k+1}^\mu\backslash X_k)=\bigcup_{0\leq k\leq n-1}R_{k+1}^\mu=\{1,2,\cdots,r\}$. Consequently, $w_1=w_2$.

Let $(Y_0,\cdots,Y_{n-1})\in\msY$.  We write $Y_i=\{k_{i,1},\cdots,k_{i,\al_i}\}$ and $R_{i+1}^\mu\backslash Y_i=\{k_{i,\al_i+1},\cdots,k_{i,\mu_{i+1}}\}$ where $k_{i,1}<k_{i,2}\cdots<k_{i,\al_i}$ and $k_{i,\al_i+1}<k_{i,\al_i+2}<\cdots<k_{i,\mu_{i+1}}$ for all $i$. Define $w\in\affSr$ by letting $w^{-1}(\mu_{0,i}+s)=k_{i,s}$ for $0\leq i\leq n-1$ and $1\leq s\leq\mu_{i+1}$. Then $w\in\afmsD_\dt\cap\fS_\mu$ and $\vartheta(w)=(Y_0,Y_1,\cdots,Y_{n-1})$.
This finishes the proof.
\end{proof}

Let $\afSrmbz=\afSr\ot\mbz$ and $\afSrmbq=\afSr\ot\mbq$, where $\mbz$ and $\mbq$ are regarded as $\sZ$-modules by specializing $\up$ to $1$. We will identify $\afSrmbz$ as a subalgebra of $\afSrmbq$. For $A\in\afThnr$ we will denote $[A]\ot 1$ by $[A]_1$.

There is a natural map
\begin{equation}\label{ti}
\ti\ :\afThn\ra\afThn\;\;A=(a_{i,j})\longmapsto\ti A=(\ti a_{i,j}),
\end{equation}
where $\ti a_{i,j}=a_{i-1,j}$ for all $i,j\in\mbz$.
We now give some multiplication formulas in the affine Schur algebra $\afSrmbz$ over $\mbz$, which are the affine version of \cite[3.1]{BLM}.

\begin{Prop}\label{funtmental multiplication formula}
Let  $A\in\afThnr$ and $\mu=\ro(A)$.
Assume $\bt\in\afmbnn$ is such that $\bt\leq\mu$.
Let $\al=\sum_{1\leq i\leq n}(\mu_i-\bt_i)\afbse_{i-1}$ and $\ga=\mu-\bt$. Then  in $\afSrmbz$

$$(1)\ \bigg[\sum\limits_{1\leq i\leq n}\al_i\afE_{i,i+1}+\diag(\beta)\bigg]_1 [A]_1=\sum\limits_{T\in\afThn,\,\ro(T)=\al\atop
a_{i,j}-t_{i-1,j}+t_{i,j}\geq0,\,\forall i,j
}\prod\limits_{1\leq i\leq n\atop j\in\mbz}\bigg({a_{i,j}-t_{i-1,j}+t_{i,j}\atop t_{i,j}}\bigg)[A+T-\ti T]_1;$$

$$(2)\ \bigg[\sum_{1\leq i\leq n}\ga_i\afE_{i+1,i}+\diag(\beta)\bigg]_1[A]_1
=\sum\limits_{T\in\afThn,\,\ro(T)=\ga\atop
a_{i,j}+t_{i-1,j}-t_{i,j},\,\forall i,j}\prod\limits_{1\leq i\leq n\atop j\in\mbz}\bigg({a_{i,j}+t_{i-1,j}-t_{i,j}\atop t_{i-1,j}}\bigg)[A-T+\ti T]_1.$$
\end{Prop}
\begin{proof}
We only prove (1). The proof for (2) is entirely similar.

Let $B=\sum_{1\leq i\leq n}\al_i\afE_{i,i+1}+\diag(\bt)$,
$\la=\ro(B)$,  $\nu=\co(A)$. Assume $d_1\in\msD^\vtg_{\la,\mu}$ and $d_2\in\msD^\vtg_{\mu,\nu}$ are such that
$\jmath_\vtg(\la, d_1,\mu)=B$ and $\jmath_\vtg(\mu, d_2,\nu)=A$.
Clearly, by \eqref{minimal coset representative}, we have
\begin{equation}\label{minimal coset representative corresponding semisimple module}
d_1(i)=-\al_n+i \text{ for } 1\leq i\leq r.
\end{equation}

By  \ref{double coset} and \eqref{conjugate intersection},
\[
\begin{split}
[B]_1[A]_1(\ul{\frak S_\nu})&=\ul{\frak S_\la d_1\frak S_\mu}\cdot d_2\cdot\ul{\afmsD_\og\cap\frak S_\nu}\\
&=\frac{1}{|\frak S_\mu|}\ul{\frak S_\la d_1\frak
S_\mu}\cdot\ul{\frak S_\mu d_2\frak S_\nu}\\
&=\frac{1}{|\frak S_\mu|}\prod_{1\leq i\leq n\atop
j\in\mbz}\frac{1}{a_{i,j}!}\ul{\frak S_\la d_1\frak
S_\mu}\cdot\ul{\frak S_\mu}\cdot d_2\cdot\ul{\frak S_\nu}\\
&=\prod_{1\leq i\leq n\atop j\in\mbz}\frac{1}{a_{i,j}!}\ul{\frak
S_\la d_1\frak S_\mu}\cdot d_2\cdot\ul{\frak S_\nu}\\
&=\prod_{1\leq i\leq n\atop j\in\mbz}\frac{1}{a_{i,j}!}\ul{\frak
S_\la}\cdot d_1\cdot\ul{\afmsD_\dt\cap\frak S_\mu}\cdot
d_2\cdot\ul{\frak S_\nu}\\
&=\sum_{w\in\afmsD_\dt\cap\frak S_\mu}\prod_{1\leq i\leq n\atop j\in\mbz}\frac{c_{i,j}^{(w)}!}{a_{i,j}!}[C^{(w)}]_1(\ul{\fS_\nu})
\end{split}
\]
where $\frak S_\og=d_2^{-1}\frak S_\mu d_2\cap\frak S_\nu$,
$\frak S_\dt=d_1^{-1}\frak S_\la d_1\cap\fS_\mu$ with $\dt=(\al_n,\bt_1,\al_1,\bt_2,\cdots,\al_{n-1},\bt_n)$ and $C^{(w)}=(c_{i,j}^{(w)})$ with $c_{i,j}^{(w)}=|R_i^\la\cap d_1wd_2R_j^\nu|$. Thus we have
\begin{equation}\label{reductive equation for [B][A]}
[B]_1[A]_1=\sum_{w\in\afmsD_\dt\cap\frak S_\mu}\prod_{1\leq i\leq n\atop j\in\mbz}\frac{c_{i,j}^{(w)}!}{a_{i,j}!}[C^{(w)}]_1
\end{equation}

Let us compute $C^{(w)}$ for $w\in\afmsD_\dt\cap\frak S_\mu$ as follows. Since, by \eqref{minimal coset representative corresponding semisimple module}, $d_1^{-1}(j)=\al_n+j$ for $1\leq j\leq r$,   we have
\begin{equation*}
\begin{split}
d_1^{-1}R_i^{\la}&=\al_n+R_i^\la=\{\mu_{0,i-1}
+\al_{i-1}+1,\mu_{0,i-1}+\al_{i-1}+2,\cdots,\mu_{0,i}+\al_i\}
=(R_i^\mu\backslash X_{i-1})\cup X_i\\
\end{split}
\end{equation*}
for $1\leq i\leq n$,
where $X_i=
\{\mu_{0,i}+1,\mu_{0,i}+2,\cdots,\mu_{0,i}+\al_i\}.$
Thus, since  $\jmath_\vtg(\mu, d_2,\nu)=A$,
we have
\begin{equation}
\begin{split}
c_{i,j}^{(w)}=|R_i^\la\cap d_1 w d_2R_j^\nu|=|w^{-1}d_1^{-1}R_i^\la\cap d_2 R_j^\nu|=a_{i,j}-t_{i-1,j}^{(w)}+t_{i,j}^{(w)}
\end{split}
\end{equation}
for $w\in\afmsD_\dt\cap\frak S_\mu$, $1\leq i\leq n$ and $j\in\mbz$,
where $t_{i,j}^{(w)}=|w^{-1}X_i\cap d_2 R_j^\nu|$. Note that, for $w\in\afmsD_\dt\cap\fS_\mu$, the numbers $t_{i,j}^{(w)}$ with $1\leq i\leq n$ and $j\in\mbz$ determine an unique matrix $T^{(w)}=(t_{i,j}^{(w)})_{i,j\in\mbz}$ in $\afThn$ by letting $t_{i+kn,j+kn}^{(w)}=t_{i,j}^{(w)}$. Consequently,
$$C^{(w)}=A+T^{(w)}-\ti T^{(w)}$$
for  $w\in\afmsD_\dt\cap\fS_\mu$.

Now, by \eqref{reductive equation for [B][A]} and noting $\ro(T^{(w)})=\al$ for   $w\in\afmsD_\dt\cap\fS_\mu$,
\begin{equation}\label{reductive equation1 for [B][A]}
\begin{split}
[B]_1[A]_1&=\sum_{w\in\afmsD_\dt\cap\frak S_\mu}\prod_{1\leq i\leq n\atop j\in\mbz}\frac{(a_{i,j}-t_{i-1,j}^{(w)}+t_{i,j}^{(w)})!}{a_{i,j}!}[A+T^{(w)}-\ti T^{(w)}]_1\\
&=\sum_{T\in\afThn,\,\ro(T)=\al\atop a_{i,j}-t_{i-1,j} +t_{i,j}\geq 0,\,\forall i,j} |\msX(T)|\prod_{1\leq i\leq n\atop j\in\mbz}\frac{(a_{i,j}-t_{i-1,j} +t_{i,j})!}{a_{i,j}!}[A+T-\ti T]_1
\end{split}
\end{equation}
where $\msX(T)=\{w\in\afmsD_\dt\cap\fS_\mu\mid T^{(w)}=T\}$. By restriction,  for each $T$, the bijective map $\vartheta$ defined in \ref{vartheta} induces a bijective map $\vartheta_T:\msX(T)\ra\msY(T)$, where
$$\msY(T)=\{(Y_0,\cdots,Y_{n-1})\in\msY\mid |Y_i\cap d_2R_j^\nu|=t_{i,j},\ \text{for}\ 0\leq i\leq n-1,\,j\in\mbz\}.$$
Furthermore, for each $T$ there is a natural bijective map $\kappa:\msY(T)\ra\msZ(T)$ defined by sending $(Y_0,\cdots,Y_{n-1})$ to $(Y_i\cap d_2R_j^\nu)_{0\leq i\leq n-1,\,j\in\mbz}$, where $$\msZ(T)=\{(Z_{i,j})_{0\leq i\leq n-1,\,j\in\mbz}\mid |Z_{i,j}|=t_{i,j},\,Z_{i,j}\han R_{i+1}^\mu\cap d_2R_j^\nu,\ \text{for}\ 0\leq i\leq n-1,\,j\in\mbz\}.$$
Consequently, $$|\msX(T)|=|\msY(T)|=|\msZ(T)|=\prod_{0\leq i\leq n-1\atop j\in\mbz}\bigg({a_{i+1,j}\atop t_{i,j}}\bigg)=\prod_{1\leq i\leq n\atop j\in\mbz}\bigg({a_{i,j}\atop t_{i-1,j}}\bigg).$$
Thus, by \eqref{reductive equation1 for [B][A]} and noting $\prod_{1\leq i\leq n,\,j\in\mbz}t_{i,j}!=\prod_{1\leq i\leq n,\,j\in\mbz}t_{i-1,j}!$,
\begin{equation*}
\begin{split}
[B]_1[A]_1
&=\sum_{T\in\afThn,\,\ro(T)=\al\atop a_{i,j}-t_{i-1,j}+t_{i,j}\geq 0,\,\forall i,j} \prod_{1\leq i\leq n\atop j\in\mbz}\frac{(a_{i,j}-t_{i-1,j} +t_{i,j})!}{t_{i-1,j}!\cdot (a_{i,j}-t_{i-1,j})!}[A+T-\ti T]_1\\
&= \sum_{T\in\afThn,\,\ro(T)=\al\atop a_{i,j}-t_{i-1,j}+t_{i,j}\geq 0,\,\forall i,j}\prod_{1\leq i\leq n\atop j\in\mbz}\bigg({a_{i,j}-t_{i-1,j}+t_{i,j}\atop t_{i,j}}\bigg)[A+T-\ti T]_1,
\end{split}
\end{equation*}
proving (1).
\end{proof}

For $A\in\aftiThn$ and $a\in\mbz$ we set
$${}_aA=A+aI$$
where $I\in\afThn$ is the identity matrix. Note that if $a$ is large enough, we have ${}_aA\in\afThn$.

Let $x$ be an indeterminate. We denote by $\sZ_1$ the subring of $\mbq[x]$ generated by $1$ and $\big({a+x\atop t}\big)$ for $a\in\mbz$ and $t\in\mbn$.

For $T\in\afThn$ and $A\in\aftiThn$ let $$P_{T,A}(x)=\prod_{1\leq i\leq n\atop j\in\mbz,\,j\not=i}\bigg({a_{i,j}-t_{i-1,j}+t_{i,j}\atop t_{i,j}}\bigg)\prod_{1\leq i\leq n}\bigg({a_{i,i}-t_{i-1,i}+t_{i,i}+x\atop t_{i,i}}\bigg)\in\sZ_1$$
and
$$Q_{T,A}(x)=\prod_{1\leq i\leq n\atop j\in\mbz,\,j\not=i}\bigg({a_{i,j}+t_{i-1,j}-t_{i,j}\atop t_{i-1,j}}\bigg)\prod_{1\leq i\leq n}\bigg({a_{i,i}+t_{i-1,i}-t_{i,i}+x\atop t_{i-1,i}}\bigg)\in\sZ_1.$$
From \ref{funtmental multiplication formula}, we immediately have the following result.
\begin{Coro}\label{stabilization property1}
Let $A,B\in\aftiThn$ is such that $\co(B)=\ro(A)$ and let $b=\sg(A)=\sg(B)$.

$(1)$ If $B-\sum_{1\leq i\leq n}\al_i\afE_{i,i+1}$ is diagonal for some $\al\in\afmbnn$. Then for large $a$ and $r=an+b$, we have in $\afSrmbz$,
$$[{}_aB]_1[{}_aA]_1=\sum_{T\in\afThn,\,\ro(T)=\al\atop a_{i,j}-t_{i-1,j}+t_{i,j}\geq 0,\,\forall i\not=j}P_{T,A}(a)[{}_a(A+T-\ti T)]_1.$$

$(2)$ If $B-\sum_{1\leq i\leq n}\al_i\afE_{i+1,i}$ is diagonal for some $\al\in\afmbnn$. Then for large $a$ and $r=an+b$, we have in $\afSrmbz$,
$$[{}_aB]_1[{}_aA]_1=\sum_{T\in\afThn,\,\ro(T)=\al\atop a_{i,j}+t_{i-1,j}-t_{i,j}\geq 0,\,\forall i\not=j}Q_{T,A}(a)[{}_a(A-T+\ti T)]_1.$$
\end{Coro}

\section{The algebra $\sK$}
We now use \ref{generalization of BLM 3.9} and \ref{stabilization property1} to derive the stabilization property of
multiplication for affine Schur algebras, which is the affine analogue of \cite[4.2]{BLM}. This property allow us to construct an algebra $\sK$ without unity.

For $A\in\afThn$, define
(cf. \cite{BLM})
$$|\!|A|\!|=\sum_{1\leq i\leq n\atop
i<j}\frac{(j-i)(j-i+1)}{2}a_{i,j}+\sum_{1\leq i\leq n\atop
i>j}\frac{(i-j)(i-j+1)}{2}a_{i,j}.$$
Let
$\aftiThn^{ss}$ be the set of $X\in\aftiThn$ such that either $X-\sum_{1\leq i\leq n}\al_i\afE_{i,i+1}$ or $X-\sum_{1\leq i\leq n}\al_i\afE_{i+1,i}$ is diagonal for some $\al\in\afmbnn$ and let
$\afThn^{ss}=\aftiThn^{ss}\cap\afThn$.

\begin{Prop}\label{stabilization property}
Let $A,B\in\aftiThn$ such that $\co(B)=\ro(A)$ Then there exist unique $X_1,\cdots,X_m\in\aftiThn$, unique $P_1(x),\cdots,P_m(x)\in\sZ_1$ and an integer $a_0\geq 0$ such that, for all $a\geq a_0$,
\begin{equation}\label{eq stabilization}
[{}_aB]_1[{}_aA]_1=\sum_{1\leq i\leq m}P_i(a)[{}_aX_i]_1
\end{equation}
in the affine Schur algebra ${\mathcal S}_{\vtg}(\cycn,an+\sg(A))_\mbz$.
\end{Prop}
\begin{proof}
If $B\in\aftiThn^{ss}$
then the assertion follows from \ref{stabilization property1}. So, by induction, if $B_1,\cdots,B_l\in\aftiThn^{ss}$ are such that $\co(B_i)=\ro(B_{i+1})$ and $\co(B_l)=\ro(A)$, then there exist $Y_j\in\aftiThn$, $Q_j(x)\in\sZ_1$ ($1\leq j\leq m$), and $a_0\in\mbn$ such that
\begin{equation}\label{eq 1-stabilization}
[{}_aB_1]_1\cdots[{}_aB_l]_1[{}_aA]_1=\sum_{1\leq j\leq m}Q_j(a)[{}_aY_j]_1
\end{equation}
for all $a\geq a_0$.

In general, we apply induction on $|\!|B|\!|$. If $|\!|B|\!|=0$ then $B$ is diagonal and $[{}_aB]_1[{}_aA]_1=[{}_aA]_1$ for all large enough $a$. Assume $|\!|B|\!|\geq 1$ and the result is true for those $|\!|B_1|\!|$ with $|\!|B_1|\!|<|\!|B|\!|$. Choose $b\in\mbn$ such that ${}_bB\in\afThn$ and apply \ref{generalization of BLM 3.9} to ${}_bB$. Thus, there exist $B_1,\cdots,B_N\in\afThn^{ss}$, such that $\co(B_i)=\ro(B_{i+1})$ and
$$[{}_aB_1]_1\cdots[{}_aB_N]_1=[{}_{a+b}B]_1+\text{lower terms}$$
for $a\in\mbn$. Let $A_i=B_i-bI\in\aftiThn^{ss}$ for $1\leq i\leq N$. Then we have
\begin{equation}\label{eq 2-stabilization}
[{}_{c}A_1]_1\cdots[{}_{c}A_N]_1=[{}_{c-b}B_1]_1\cdots[{}_{c-b}B_N]_1
=[{}_{c}B]_1+\text{lower terms}
\end{equation}
for $c\geq b$. By \eqref{eq 1-stabilization}, there exist $Z_i,Z_j'\in\aftiThn$ and $Q_i(x),Q_j'(x)\in\sZ_1$ ($1\leq i\leq m$,  $1\leq j\leq m'$) such that
\begin{equation}\label{eq 3-stabilization}
[{}_{c}A_1]_1\cdots[{}_{c}A_N]_1=\sum_{1\leq i\leq m}Q_i(c)[{}_cZ_i]_1,
\end{equation}
and
\begin{equation}\label{eq 4-stabilization}
[{}_{c}A_1]_1\cdots[{}_{c}A_N]_1[{}_cA]_1=\sum_{1\leq j\leq m'}Q_j'(c)[{}_cZ_j']_1
\end{equation}
for all large enough $c$.
Comparing \eqref{eq 2-stabilization} with \eqref{eq 3-stabilization}, we see that we may assume that $Z_1=B$, $Q_1(x)=1$ and $Z_i\p B$ for $2\leq i\leq m$. Thus by \eqref{eq 3-stabilization} and \eqref{eq 4-stabilization}, for large $c$,
\begin{equation}
\begin{split}
[{}_cB]_1[{}_cA]_1&=[{}_{c}A_1]_1\cdots[{}_{c}A_N]_1[{}_cA]_1-\sum_{2\leq i\leq m}Q_i(c)[{}_cZ_i]_1[{}_cA]_1\\
&=\sum_{1\leq j\leq m'}Q_j'(c)[{}_cZ_j']_1-\sum_{2\leq i\leq m}Q_i(c)[{}_cZ_i]_1[{}_cA]_1.
\end{split}
\end{equation}
Since, by \cite[3.7.6]{DDF},  $|\!|Z_i|\!|<|\!|B|\!|$ for $i>1$, the induction hypothesis applied to $Z_i$ shows that
$[{}_cZ_i]_1[{}_cA]_1$ is given by an expression like in the right hand side of \eqref{eq stabilization},
for $i>1$ and large $c$. Consequently,  $[{}_cB]_1[{}_cA]_1$ is of the required form.
\end{proof}

We now use \ref{stabilization property} to construct the $\mbz$-algebra $\sK$ as follows.
Let $\tisK$ be the free $\sZ_1$-module with basis $\{A\mid A\in\aftiThn\}$. There is a unique structure of associative $\sZ_1$-algebra (without unit) on this module in which
$B\cdot A=\sum_{1\leq i\leq m}P_i(x)X_i$ (notation of \ref{stabilization property}) if $\co(B)=\ro(A)$ and $B\cdot A=0$, otherwise.
Consider the specialization $\sZ_1\ra\mbz$ obtained by sending $x$ to $0$ and let
$$\sK=\tisK\ot_{\sZ_1}\mbz.$$
Then $\sK$ is an associative algebra over $\mbz$ with basis $\{A\ot 1\mid A\in\aftiThn\}$. We will denote the element $A\ot 1$ by $[A]_1$ in the sequel. By \ref{stabilization property1} and \ref{generalization of BLM 3.9} the following multiplication
formulas hold in the algebra $\sK$.
\begin{Prop}\label{formula in K}
Let $A,B\in\aftiThn$ be such that $\co(B)=\ro(A)$. In the algebra $\sK$, the following
statements hold.

$(1)$ If $B-\sum_{1\leq i\leq n}\al_i\afE_{i,i+1}$ is diagonal for some $\al\in\afmbnn$, then
\begin{equation*}
[B]_1 [A]_1=\sum\limits_{T\in\afThn,\,\ro(T)=\al\atop a_{i,j}-t_{i-1,j}+t_{i,j}\geq 0,\,\forall i\not=j}\prod\limits_{1\leq i\leq n\atop j\in\mbz}\bigg({a_{i,j}-t_{i-1,j}+t_{i,j}\atop t_{i,j}}\bigg)[A+T-\ti T]_1;
\end{equation*}

$(2)$ If $B-\sum_{1\leq i\leq n}\al_i\afE_{i+1,i}$ is diagonal for some $\al\in\afmbnn$, then
$$ [B]_1[A]_1
=\sum\limits_{T\in\afThn,\,\ro(T)=\al\atop a_{i,j}+t_{i-1,j}-t_{i,j}\geq 0,\,\forall i\not=j}\prod\limits_{1\leq i\leq n\atop j\in\mbz}\bigg({a_{i,j}+t_{i-1,j}-t_{i,j}\atop t_{i-1,j}}\bigg)[A-T+\ti T]_1.$$

$(3)$ There exist upper triangular matrixs $A_i$ and lower triangular matrixs $B_j$ in $\aftiThn^{ss}$ ($1\leq i\leq s$, $1\leq j\leq t$) such that
\begin{equation*}
[A_1]_1
\cdots
[A_s]_1
[B_1]_1
\cdots
[B_t]_1=[A]_1+\text{lower terms},
\end{equation*}
where ``lower
terms" stands for a $\mbz$-linear combination of terms $[A']_1$ with $A'\in\aftiThn$, $A'\p A$, $\co(A')=\co(A)$ and $\ro(A')=\ro(A)$.
\end{Prop}

The algebra $\sK$ and $\afSrmbz$ are related by the following algebra homomorphism.
\begin{Prop}\label{dzr}
The linear map $\dzr:\sK\ra\afSrmbz$ defined by
$$\dzr([A]_1)=\begin{cases}[A]_1& \mathrm{if\ }A\in\afThnr;\\
0&  \mathrm{otherwise}\end{cases}$$
is a surjective  algebra homomorphism.
\end{Prop}
\begin{proof}
Since,  by \ref{formula in K}(3), the algebra $\sK$ is generated by $[A]$ with $A\in\aftiThn^{ss}$, it is enough to prove that
\begin{equation}\label{eq dzr}
\dzr([B]_1[A]_1)=\dzr([B]_1)\dzr([A]_1)
\end{equation}
for $B\in\aftiThn^{ss}$ and $A\in\aftiThn$ with $\co(B)=\ro(A)$.
If $\sg(A)\not=r$, then $\dzr([B]_1[A]_1)=0=\dzr([B]_1)\dzr([A]_1)$.

Now we assume $\sg(A)=r$, $\co(B)=\ro(A)$, $B=\sum_{1\leq i\leq n}\al_i\afE_{i,i+1}+\diag(\bt)\in\aftiThn^{ss}$ for some $\al\in\afmbnn$ and $\bt\in\afmbzn$. Let us prove \eqref{eq dzr} in three cases.

Case 1 If $A,B\in\afThnr$, then the assertion follows from
\ref{funtmental multiplication formula} and \ref{formula in K}(1).

Case 2 Suppose $a_{i_0,i_0}<0$ for some $1\leq i_0\leq n$. If $T\in\afThn$ is such that $\ro(T)=\al$ and $A+T-\ti T\in\aftiThn$, then $$\dzr\left(\bigg({a_{i_0,i_0}+t_{i_0,i_0}-t_{i_0-1,i_0}\atop t_{i_0,i_0}}\bigg)[A+T-\ti T]_1\right)=0$$
since $a_{i_0,i_0}+t_{i_0,i_0}-t_{i_0-1,i_0}<t_{i_0,i_0}$. It follows from \ref{formula in K}(1) that $\dzr([B]_1[A]_1)=0=\dzr([B]_1)\dzr([A]_1)$.

Case 3 Suppose $\bt_{i_0}<0$ for some $1\leq i_0\leq n$.
Let $T\in\afThn$ be such that $\ro(T)=\al$ and $A+T-\ti T\in\aftiThn$.
Since $\bt+\sum_{1\leq i\leq n}\al_i\afbse_{i+1}=\co(B)=\ro(A)$ and $\ro(T)=\al$,  we have  $$\sum_{s\in\mbz}(a_{i_0,s}-t_{i_0-1,s})=\sum_{s\in\mbz}a_{i_0,s}-\al_{i_0-1}
=\bt_{i_0}<0,$$
and hence $a_{i_0,k}-t_{i_0-1,k}+t_{i_0,k}<t_{i_0,k}$ for some $k\in\mbz$. Thus, $$\dzr\left(\bigg({a_{i_0,k}+t_{i_0,k}-t_{i_0-1,k}\atop t_{i_0,k}}\bigg)[A+T-\ti T]_1\right)=0.$$
This together with \ref{formula in K}(1) implies that $\dzr([B]_1[A]_1)=0=\dzr([B]_1)\dzr([A]_1)$.

Similarly, if $B-\sum_{1\leq i\leq n}\al_i\afE_{i+1,i}$ is diagonal for some $\al\in\afThn$,   $\dzr([B]_1[A]_1)=0=\dzr([B]_1)\dzr([A]_1)$. The proof is completed.
\end{proof}

\section{The completion algebra $\hsKq$ of $\sK$ and multiplication formulas}
Let $\sKq=\sK\ot_\mbz\mbq$.
As in \cite[5.1]{BLM}, let $\hsKq$ be
the vector space of all formal (possibly infinite) $\mbq$-linear combinations
$\sum_{A\in\aftiThn}\beta_A[A]$ such that
for any ${\bf x}\in\mathbb Z^n$, the sets
${\{A\in\aftiThn\ |\ \beta_A\neq0,\ \ro(A)={\bf
x}\}}$ and ${\{A\in\aftiThn\ |\ \beta_A\neq0,\ \co(A)={\bf
x}\}}$ are finite.
We shall regard $\sK$ naturally as a subset of $\hsKq$. We
can define the product of two elements
$\sum_{A\in\widetilde\Xi}\beta_A[A]_1$,
$\sum_{B\in\widetilde\Xi}\gamma_B[B]_1$ in $\hsKq$ to be
$\sum_{A,B}\beta_A\gamma_B[A]_1\cdot[B]_1$ where $[A]_1\cdot[B]_1$ is the
product in $\sK$.
This defines an associative algebra structure on $\hsKq$.
This algebra has a unit element: $\sum_{\la\in\afmbzn}[\diag(\la)]_1$.

We now establish  some important multiplication formulas in $\hsKq$ and $\afSrmbq$, which will be used to realize $\afuglz$ as a $\mbz$-subalgebra of $\hsKq$. These formulas are the affine analogue of \cite[5.3]{BLM}.

For $m,k_1,k_2,\cdots,k_t\in\mbn$ with $\sum_{1\leq i\leq l}k_i=m$, let  $\left({m\atop k_1,k_2,\cdots,k_t}\right)=\frac{m!}{k_1!k_2!\cdots k_t!}.$ For $\la,\mu^{(1)},\cdots,\mu^{(s)}\in\afmbnn$ with $\la=\sum_{1\leq j\leq s}\mu^{(j)}$, let $$\bpa{\la}{\mu^{(1)},\cdots,\mu^{(s)}}=\prod_{1\leq i\leq n}\bpa{\la_i}{\mu_i^{(1)},\cdots,\mu_i^{(s)}}.$$
Recall the order relation $\leq$ defined in \eqref{order on afmbzn}.
We need the following well known combinational formulas.
\begin{Lem}\label{combination formula1}
For $m,n\in\mbz$, $a,b\in\mbn$ we have

$(1)$ $\left({n\atop a}\right)=\sum\limits_{0\leq j\leq a}\left({m\atop j}\right)\left({n-m\atop a-j}\right);$

$(2)$ $\left({m\atop a}\right)\left({m\atop b}\right)=\sum\limits_{0\leq c\leq \min\{a,b\}}\left({a+b-c\atop c,a-c,b-c}\right)\left({m\atop a+b-c}\right).$
\end{Lem}

\begin{Coro}\label{combination formula2}
For $\la,\mu\in\afmbnn$ and $\al,\bt\in\afmbzn$ we have

$(1)\left({\al+\bt\atop\la}\right)=\sum\limits_{\mu\in\afmbnn,\,\mu\leq\la}
\left({\al\atop\mu}\right)\left({\bt\atop\la-\mu}\right);$

$(2)\left({\al\atop\la}\right)\left({\al\atop\mu}\right)=
\sum\limits_{\ga\in\afmbnn\atop\ga\leq\la,\,\ga\leq\mu}
\left({\la+\mu-\ga\atop\ga,\la-\ga,\mu-\ga}\right)\left({\al\atop\la+ \mu-\ga}\right)$.
\end{Coro}

For $A\in\afThnpm$, $\la\in\afmbnn$ let
\begin{equation*}
\begin{split}
A\br\la&=\sum_{\mu\in\afmbzn}\bigg({\mu\atop\la}\bigg)
[A+\diag(\mu)]_1\in\hsKq\\
A\br{\la,r}&=\sum_{\mu\in\afLa(n,r-\sg(A))}\bigg({\mu\atop\la}\bigg)
[A+\diag(\mu)]_1\in\afSrmbz.
\end{split}
\end{equation*}
Also, for $A\in\afMnz$, define
\begin{center}
$A\br\la=0$ and $A\br{\la,r}=0$ if $a_{i,j}<0$ for some $i\not=j$.
\end{center}
For $A\in\aftiThn$ let $\la(A)$ be the element in $\afmbzn$ such that $$\diag(\la(A))=A^0,$$
where $A^0$ is defined in \eqref{A^+,A^-,A^0}. Recall the map $\ti\ :\afThn\ra\afThn$ defined in \eqref{ti}.
\begin{Prop}\label{formula in Vz}
Let $A\in\afThnpm$, $B=\sum_{1\leq i\leq n}\al_i\afE_{i,i+1}$ and
$C=\sum_{1\leq i\leq n}\al_i\afE_{i+1,i}$ with $\al\in\afmbnn$. Let $\la,\mu\in\afmbnn$. The following identities holds in $\hsKq$:
\begin{equation*}
\begin{split}
(1)
0\br\mu A\br\la&=\sum_{\dt\in\afmbnn\atop\dt\leq\mu}\left(
\sum_{\bt\in\afmbnn\atop\bt\leq\mu-\dt,\,\bt\leq\la}
\left({\ro(A)\atop\mu-\bt-\dt}
\right)\left({\dt+\la\atop\bt,\dt,\la-\bt}\right)\right)A\br{\la+\dt}\\
(2) B\br\bfl A\br\la
&=\sum_{T\in\afThn,\,\dt\in\afmbnn\atop \ro(T)=\al,\,\dt\leq\la}
a(T,\dt)\prod_{1\leq i\leq n\atop j\in\mbz,\,j\not= i}\bigg({a_{i,j}-t_{i-1,j}+t_{i,j}\atop t_{i,j}}\bigg)
(A+T^\pm-\ti T^\pm)\br{\la(T)+\dt}
\end{split}
\end{equation*}
where $$a(T,\dt)=\sum_{\bt\in\afmbnn,\,\bt\leq\la(T),\atop\bt\leq\la-\dt}
\left({\la(\ti T)-\la(T)\atop\la-\bt-\dt}\right)
\left({\la(T)+\dt\atop\bt,\dt,\la(T)-\bt}\right)\in\mbz;$$
\begin{equation*}
\begin{split}
(3) C\br\bfl A\br\la
&=\sum_{T\in\afThn,\,\dt\in\afmbnn\atop \ro(T)=\al,\,\dt\leq\la}
b(T,\dt)\prod_{1\leq i\leq n\atop j\in\mbz,\,j\not= i}\bigg({a_{i,j}+t_{i-1,j}-t_{i,j}\atop t_{i-1,j}}\bigg)
(A-T^\pm+\ti T^\pm)\br{\la(\ti T)+\dt}
\end{split}
\end{equation*}
where $$b(T,\dt)=\sum_{\bt\in\afmbnn,\,\bt\leq\la(\ti T),\atop\bt\leq\la-\dt}
\left({\la(T)-\la(\ti T)\atop\la-\bt-\dt}\right)
\left({\la(\ti T)+\dt\atop\bt,\dt,\la(\ti T)-\bt}\right)\in\mbz.$$
The same formulas hold in $\afSrmbz$ with $A\{\la\}$ replaced by $A\{\la,r\}$.
\end{Prop}
\begin{proof}
First, by \ref{combination formula2},
\begin{equation*}
\begin{split}
0\br\mu A\br\la&=
\sum_{\al\in\afmbzn}\bpa{\ro(A)+\al}{\mu}\bpa{\al}{\la}[A+\diag(\al)]_1\\
&=\sum_{\al\in\afmbzn,\,\bfj\in\afmbnn\atop\bfj\leq\mu}\bpa{\ro(A)}{\mu-\bfj}\bpa{\al}{\bfj}
\bpa{\al}{\la}[A+\diag(\al)]_1\\
&=\sum_{\al\in\afmbzn,\,\bfj,\bt\in\afmbnn\atop
\bt\leq\bfj\leq\mu,\,\bt\leq\la}\bpa{\ro(A)}{\mu-\bfj}
\bpa{\bfj+\la-\bt}{\bt,\bfj-\bt,\la-\bt} \bpa{\al}{\bfj+\la-\bt}[A+\diag(\al)]_1\\
&=\sum_{\bfj,\bt\in\afmbnn\atop
\bt\leq\bfj\leq\mu,\,\bt\leq\la}\bpa{\ro(A)}{\mu-\bfj}
\bpa{\la+\bfj-\bt}{\bt,\bfj-\bt,\la-\bt} A\br{\la+\bfj-\bt}\\
&=\sum_{\dt\in\afmbnn\atop\dt\leq\mu}\left(
\sum_{\bt\in\afmbnn\atop\bt\leq\mu-\dt,\,\bt\leq\la}
\left({\ro(A)\atop\mu-\bt-\dt}
\right)\left({\la+\dt\atop\bt,\dt,\la-\bt}\right)\right)A\br{\la+\dt},\\
\end{split}
\end{equation*}
proving (1). To prove (2), we conclude from \ref{formula in K} that
\begin{equation*}
\begin{split}
B\br\bfl A\br\la &=\sum_{\ga\in\afmbzn}\left({\ga\atop\la}\right)
\left[B+\diag\left(\ga+\ro(A)-\sum_{1\leq i\leq n}\al_i\afbse_{i+1}\right)\right]_1[A+\diag(\ga)]_1\\
&=\sum_{T\in\afThn \atop \ro(T)=\al }
\prod_{1\leq i\leq n\atop j\in\mbz,\,j\not= i}\bigg({a_{i,j}-t_{i-1,j}+t_{i,j}\atop t_{i,j}}\bigg)x_T
\end{split}
\end{equation*}
where
\begin{equation*}
\begin{split}
x_T&=
\sum_{\ga\in\afmbzn}\left({\ga\atop\la}\right)
\left({\ga-\la(\ti T)+\la(T)\atop\la(T)}\right)[A+T^\pm-\ti T^\pm+\diag(\ga-\la(\ti T)+\la(T))]_1.
\end{split}
\end{equation*}
Furthermore, by \ref{combination formula2} we have
\begin{equation*}
\begin{split}
x_T&=\sum_{\nu\in\afmbzn}\left({\nu-\la(T)+\la(\ti T)\atop\la}\right)
\left({\nu\atop\la(T)}\right)[A+T^\pm-\ti T^\pm+\diag(\nu)]_1\\
&=\sum_{\bfj\in\afmbnn,\,\bfj\leq\la\atop\nu\in\afmbzn}
\left({\la(\ti T)-\la(T)\atop\la-\bfj}\right)
\left({\nu\atop\bfj}\right)
\left({\nu\atop\la(T)}\right)[A+T^\pm-\ti T^\pm+\diag(\nu)]_1\hspace{100pt}
\\
\end{split}
\end{equation*}
\begin{equation*}
\begin{split}
&
=\sum_{\bfj,\bt\in\afmbnn\atop\bt\leq\bfj\leq\la,\;\bt\leq\la(T)} \left({\la(\ti T)-\la(T)\atop\la-\bfj}\right)
\left({\bfj+\la(T)-\bt\atop\bt,\,\bfj-\bt,\,\la(T)-\bt}\right)
\\
&\qquad\qquad\times\sum_{\nu\in\afmbzn}
\left({\nu\atop\bfj+\la(T)-\bt}\right)[A+T^\pm-\ti T^\pm+\diag(\nu)]_1\hspace{130pt}\\
\end{split}
\end{equation*}
\begin{equation*}
\begin{split}
&
=\sum_{\bfj,\bt\in\afmbnn\atop\bt\leq\bfj\leq\la,\;\bt\leq\la(T)} \left({\la(\ti T)-\la(T)\atop\la-\bfj}\right)
\left({\bfj-\bt+\la(T)\atop\bt,\,\bfj-\bt,\,\la(T)-\bt}\right)(A+T^\pm-\ti T^\pm)\br{\bfj-\bt+\la(T)}
\\
&=\sum_{\dt\in\afmbnn,\,\dt\leq\la} a(T,\dt)(A+T^\pm-\ti T^\pm)\br{\dt+\la(T)}.
\end{split}
\end{equation*}
Therefore, (2) holds. Formula (3) is proved similarly.
\end{proof}

\section{The algebra $\Vz$}
We shall denote by $\Vz$  the $\mbz$-submodule of $\hsKq$ spanned by
$$\frak{\Bz}:=\{A\br\la\mid A\in\afThnpm,\,\la\in\afmbnn\}.$$
We will prove that $\Vz$ is actually a $\mbz$-subalgebra of $\hsKq$ in \ref{Vz-subalgebra}.

\begin{Lem}\label{basis-Vz}
The set $\frak{\Bz}$ forms a $\mbz$-basis for $\Vz$.
\end{Lem}
\begin{proof}
It is enough to prove the linear independence of $\Bz$.
Suppose $$\sum_{A\in\afThnpm\atop\la\in\afmbnn}k_{A,\la} A\br\la=0,$$
for some $k_{A,\la}\in\mbz$. Then
$$0=
\sum_{A\in\afThnpm\atop\la\in\afmbnn}k_{A,\la} A\br\la
=\sum_{A\in\afThnpm\atop\mu\in\afmbnn}
\left(
\sum_{\la\in\afmbnn}k_{A,\la}
\bigg({\mu\atop\la}\bigg)\right)[A+\diag(\mu)]$$
for some $k_{A,\la}\in\mbz$. Thus,
$\sum_{\la\in\afmbnn}k_{A,\la}
\left({\mu\atop\la}\right)=0,$
for any $A\in\afThnpm$, $\mu\in\afmbzn$.
We want to show that $k_{A,\la}=0$ for all $A,\la$. If this is not the case, then there exist $B\in\afThnpm$ such that $X_B:=\{\la\in\afmbnn\mid k_{B,\la}\not=0\}\not=\varnothing$. Let $\nu$ be the minimal element in $X_B$ with respect to the lexicographic order.  Then
$$k_{B,\nu}=\sum_{\la\in X_B\atop\nu\geq\la}k_{B,\la}
\left({\nu\atop\la}\right)=\sum_{\la\in\afmbnn}k_{B,\la}
\left({\nu\atop\la}\right)=0$$
since $\nu$ is minimal. This is a contradiction.
\end{proof}

Let $\Vzp=\spann_\mbz\{A\br\bfl\mid A\in\afThnp\}$,
$\Vzm=\spann_\mbz\{A\br\bfl\mid A\in\afThnm\}$ and $\Vzz=\spann_\mbz\{0\br\la\mid \la\in\afmbnn\}$.
By \ref{formula in Vz}(1),
$\Vzz$ is a $\mbz$-subalgebra of $\hsKq$.
Also, we will see in \ref{Vzp-Vzm} that $\Vzp$ and $\Vzm$ are  $\mbz$-subalgebras of $\hsKq$.

The maps $\tau_r$ defined in \eqref{taur} induce an algebra anti-automorphism
$$\tau:\sK\ra\sK\quad([A]_1\ra [\tA]_1).$$
Consequently, we get an algebra anti-automorphism
\begin{equation}\label{htau}
\h\tau:\hsKq\ra\hsKq
\end{equation}
defined by sending $\sum_{A}\beta_A[A]_1$ to $\sum_{A}\beta_A[\tA]_1$. Clearly, $\h\tau(A\br\la)=(\tA)\br\la$ for $A\in\afThnpm$ and $\la\in\afmbnn$. Thus,
\begin{equation}\label{tau(Vzp)}
\Vzm=\h\tau(\Vzp).
\end{equation}

\begin{Lem}\label{Vzp-Vzm}
$(1)$ $\Vzp$ (resp., $\Vzm$) is a $\mbz$-subalgebras of $\hsKq$
and the linear map $\thp:\Hallz\ra\Vzp$ (resp., $\thm:\Hallzop\ra\Vzm$) taking $u_{A,1}\ra A\br\bfl$ (resp., $u_{A,1}\ra (\tA)\br\bfl$) for $A\in\afThnp$ is an algebra isomorphism.

$(2)$ $\Vzp$ (resp., $\Vzm$) is generated by $\sum_{\al\in\afmbnn}\al_i\afE_{i,i+1}\br\bfl$ (resp., $\sum_{\al\in\afmbnn}\al_i\afE_{i+1,i}\br\bfl$) for $\al\in\afmbnn$ as a $\mbz$-algebra.
\end{Lem}
\begin{proof}
Statement (2) follows from (1) and \ref{tri Hall}. We now prove (1).
Let $\tiVzp$ be the $\mbz$-submodule of $\prod_{r\geq 0}\afSrmbq$ spanned by the elements $(A\br{\bfl,r})_{r\geq 0}$ for $A\in\afThnp\}$.
Since the elements $(A\br{\bfl,r})_{r\geq 0}$ ($A\in\afThnp$) are linearly independent,
the map $\etarp$ defined in \ref{afzrpm} induce an injective algebra homomorphism $$\etap=\prod_{r\geq 0}\etarp:\Hallq\ra\prod_{r\geq 0}\afSrmbq.$$
Thus $\tiVzp=\etap(\Hallz)$ is a $\mbz$-subalgebra of $\prod_{r\geq 0}\afSrmbq$ and the restriction of $\etap$ to $\Hallz$ induces a $\mbz$-algebra isomorphism
\begin{equation}\label{etap}
\etap:\Hallz\tong\tiVzp.
\end{equation}

On the other hand, the map $\dzr$ defined in \ref{dzr} induces a surjective algebra homomorphism
\begin{equation}\label{hzr}
\hzr:\hsKq\ra\afSrmbq
\end{equation}
sending $\sum_{A\in\aftiThn}\bt_A[A]_1$ to $\sum_{A\in\aftiThn}\bt_A\dzr([A]_1)$.
Consequently, we get an algebra homomorphism
\begin{equation}
\hz=\prod_{r\geq 0}\hzr:\hsKq\ra\prod_{r\geq 0}\afSrmbq.
\end{equation}
Since $\hz(A\br\bfl)=(A\br{\bfl,r})_{r\geq 0}$ for $A\in\afThnp$
and the elements $(A\br{\bfl,r})_{r\geq 0}$ ($A\in\afThnp$) are linearly independent,
the restriction of $\hz$ to $\Vzp$ is injective and hence
we get a bijective map  $$\hz:\Vzp\ra\tiVzp.$$
This, together with \eqref{etap}, implies that $\Vzp$ is a subalgebra  of $\hsKq$ and $\thp=\hz^{-1}\circ\etap$ is an algebra isomorphism.
Finally, using \eqref{tau(Vzp)}, we get the similar result for $\Vzm$.
\end{proof}

Recall the notation $A_\bsa$ and $B_\bsa$ introduced in \eqref{Aa}.
For $w=\bsa_1\bsa_2\cdots\bsa_m\in\ti\Sg$ with the tight form
$\bsb_1^{x_1}\bsb_2^{x_2}\cdots\bsb_t^{x_t}$ we let
\begin{equation*}
\begin{split}
\ttn_{(w)}^+&=A_{x_1\bsb_1}\br\bfl A_{x_2\bsb_2}\br\bfl\cdots
A_{x_t\bsb_t}\br\bfl\in\hsKq ,\\
\ttn_{(w)}^-&=B_{x_1\bsb_1}\br\bfl B_{x_2\bsb_2}\br\bfl\cdots
B_{x_t\bsb_t}\br\bfl\in\hsKq.
\end{split}
\end{equation*}
The triangular relation for affine Schur algebras can be lifted to the level of $\hsKq$ as follows.
\begin{Lem}\label{tri-Vz}
Let $A\in\afThnpm$ and $\la\in\afmbnn$.

$(1)$ We have
$$A^+\br\bfl 0\br\la A^-\br\bfl=A\br\la+
\sum_{\bfj\in\afmbnn\atop\bfj<\la}
\left({\bfsg(A)\atop\la-\bfj}\right)A\br\bfj+f$$
where $f$ is a $\mbz$-linear combination of $B\br\nu$ such that $B\in\afThnpm$, $B\p A$ and $\nu\in\afmbnn$.

$(2)$ There exist $w_{A^+},w_{A^-}\in\ti\Sg$ such that $\wp(w_{A^+})=A^+$, $\wp(w_{A^-})=A^-$ and
$$\ttn^+_{(w_{A^+})}0\br\la\ttn^-_{(w_{A^-})}=A\br\la+
\sum_{\bfj\in\afmbnn\atop\bfj<\la}
\left({\bfsg(A)\atop\la-\bfj}\right)A\br\bfj+g$$
where $g$ is a $\mbz$-linear combination of $B\br\nu$ such that $B\in\afThnpm$, $B\p A$ and $\nu\in\afmbnn$.
\end{Lem}
\begin{proof}

By \ref{tri-affine Schur algebras} and \ref{stabilization property} for any $\mu\in\afmbzn$ we have
\begin{equation}
A^+\br\bfl[\diag(\mu)]A^-\br\bfl=[A+\diag(\mu-\bfsg(A))]+f_\mu
\end{equation}
where $f_\mu$ is a $\mbz$-linear combination of $[B]$ such that $B\p A$
and $\co(B)=\co(A)+\mu-\bfsg(A)$ and $\ro(B)=\ro(A)+\mu-\bfsg(A)$. This equality together with \ref{combination formula2} implies that
\begin{equation*}
\begin{split}
A^+\br\bfl 0\br\la A^-\br\bfl&=\sum_{\mu\in\afmbzn}\left({\mu\atop\la}\right)
([A+\diag(\mu-\bfsg(A))]+f_\mu)\\
&=\sum_{\nu\in\afmbzn}\left({\nu+\bfsg(A)\atop\la}\right)
[A+\diag(\nu)]+f\\
&=\sum_{\nu\in\afmbzn}\sum_{\bfj\in\afmbnn\atop\bfj\leq\la}
\left({\nu\atop\bfj}\right)
\left({\bfsg(A)\atop\la-\bfj}\right)
[A+\diag(\nu)]+f\\
&=\sum_{\bfj\in\afmbnn\atop\bfj\leq\la}
\left({\bfsg(A)\atop\la-\bfj}\right)
A\br\bfj+f
\end{split}
\end{equation*}
where $f=\sum_{\mu\in\afmbzn}\left({\mu\atop\la}\right)f_\mu$. By \ref{formula in Vz} and \ref{Vzp-Vzm}(2), $f$ must be a $\mbz$-linear combination of $B\br\nu$ for various $B\in\afThnpm$ such that $B\p A$ and various $\nu\in\afmbnn$. This proves (1). The assertion (2) follows from (1), \ref{tri Hall} and \ref{Vzp-Vzm}(1).
\end{proof}

\begin{Coro}\label{tri decomposition}
We have $\Vz=\Vzp\Vzz\Vzm\cong\Vzp\ot\Vzz\ot\Vzm$.
\end{Coro}
\begin{proof}
Clearly, \ref{tri-Vz}(1) implies that $\Vz=\Vzp\Vzz\Vzm$. Furthermore, by \ref{basis-Vz} and \ref{tri-Vz}(1),  the set $\{A^+\br\bfl 0\br\la A^-\br\bfl\mid A\in\afThnpm,\,\la\in\afmbnn\}$ is linearly independent. The proof is completed.
\end{proof}

Now we can prove the main result of this section, which is the affine analogue of \cite[5.5]{BLM}.
\begin{Prop}\label{Vz-subalgebra}
$(1)$ $\Vz$ is a $\mbz$-subalgebra of $\hsKq$.

$(2)$ The elements $\sum_{1\leq i\leq n}\al_i\afE_{i,i+1}\br\bfl$, $\sum_{1\leq i\leq n}\al_i\afE_{i+1,i}\br\bfl$, $0\br{\la_i\afbse_i}$ (for $\al\in\afmbnn$, $\la_i\in\mbn$, $1\leq i\leq n$) generate $\Vz$ as a $\mbz$-algebra.
\end{Prop}
\begin{proof}
Let $\Vz_1$ be the $\mbz$-subalgebra of $\hsKq$ generated by the elements indicated in (2). From \ref{formula in Vz}, we see that $\Vz_1\han\Vz_1\Vz\han\Vz$. So it is enough to prove that $A\br\la\in\Vz_1$ for all $A\in\afThnpm$ and $\la\in\afmbnn$.
We shall prove this by induction on $\ddet A$. If $\ddet A=0$, then $A=0$ and $0\br\la=0\br{\la_1\afbse_1}\cdots 0\br{\la_n\afbse_n}\in\Vz_1$.

Now we assume that $\ddet A>0$ and our statement is true
for $A'$ with $\ddet {A'}<\ddet A$.
By \ref{tri-Vz}(2), there exist $w_{A^+},w_{A^-}\in\ti\Sg$ such that
$$\ttn^+_{(w_{A^+})}\ttn^-_{(w_{A^-})}=A\br\bfl+g$$
where $g$ is a $\mbz$-linear combination of $B\br\nu$ with $B\in\afThnpm$, $B\p A$ and $\nu\in\afmbnn$. Since, by \cite[3.7.6]{DDF}, $B\p A$ implies that $|\!|B|\!|<|\!|A|\!|$, we have by the induction hypothesis $g\in\Vz_1$. Consequently, $A\br\bfl\in\Vz_1$. Furthermore,  by \ref{formula in Vz}(1),
\begin{equation}\label{0{la}A{0}}
0\br\la A\br\bfl=
A\br\la+\sum_{\dt<\la}\bpa{\ro(A)}{\la-\dt}A\br\dt=
A\br\la+\sum_{\dt<\la\atop
\sg(\dt)<\sg(\la)}\bpa{\ro(A)}{\la-\dt}A\br\dt.
\end{equation}
Thus, using induction on $\sg(\la)$, we see that $A\br\la\in\Vz_1$ for $\la\in\afmbnn$. This finishes the proof.
\end{proof}

Let $\Vq=\spann_\mbq\fB$. Then, by \ref{Vz-subalgebra}, $\Vq$ is a $\mbq$-subalgebra of $\hsKq$. We will prove that $\Vq$ is isomorphic to $\afuglq$. By \ref{basis-Vz}, the set $\fB$ forms a $\mbq$-basis for $\Vq$. We end this section with the construction of another $\mbq$-basis for $\Vq$.
For $A\in\afThnpm$ and $\bfj\in\afmbnn$, define (cf. \cite[(3.0.3)]{Fu09}, \cite[(5.4.2.1)]{DDF})
\begin{equation*}
\begin{split}
A[\bfj]&=\sum_{\la\in\afmbzn}\la^\bfj
[A+\diag(\la)]_1\in\hsKq\\
A[\bfj,r]&=\sum_{\la\in\afLa(n,r-\sg(A))}\la^\bfj
[A+\diag(\la)]_1\in\afSrmbz,
\end{split}
\end{equation*}
where $\la^\bfj=\prod_{1\leq i\leq n}\la_i^{j_i}$.
Note that, by definition,  $A[\bfl]=A\br\bfl$ and
$0[\afbse_i]=0\br{\afbse_i}$ for $A\in\afThnpm$ and $i\in\mbz$.
Clearly,  the following multiplication formula follows immediately from the definition.



\begin{Lem}\label{formula in Wq}
For $A\in\afThnpm$ and $\bfj,\bfj'\in\afmbnn$ we have
$$0[\bfj']A[\bfj]=\sum_{\al\in\afmbnn\atop\al\leq\bfj'}
\bpa{\bfj'}{\al}(\ro(A))^{\bfj'-\al}A[\al+\bfj].$$
In particular we have $0[\bfj']0[\bfj]=0[\bfj+\bfj']$.
\end{Lem}

\begin{Lem}\label{Vqz}
Let $\Vqz$ be the $\mbq$-subspace of $\Vq$ spanned by the elements $0\br{\la}$ for $\la\in\afmbnn$. Then
the set $\{0[\bfj]\mid\bfj\in\afmbnn\}$ forms a $\mbq$-basis for $\Vqz$.
\end{Lem}
\begin{proof}
Let $\Vqz_1=\spann_\mbq\{0[\bfj]\mid\bfj\in\afmbnn\}$. By \cite[6.3.3]{DDF}, it is enough to prove that $\Vqz=\Vqz_1$.
Since, by \ref{formula in Vz}(1), $\Vqz$ is a $\mbq$-subalgebra of $\hsKq$, we have
$0[\bfj]=0[\afbse_1]^{j_1}\cdots 0[\afbse_n]^{j_n}
=0\br{\afbse_1}^{j_1}\cdots 0\br{\afbse_n}^{j_n}\in\Vqz$ for $\bfj\in\afmbnn$. Furthermore, by \ref{formula in Wq}, $\Vqz_1$ is a $\mbq$-subalgebra of $\hsKq$. This implies that
$$0\br\bfj=\prod_{1\leq i\leq n}0\br{j_i\afbse_i}
=\prod_{1\leq i\leq n}\frac{0[\afbse_i](0[\afbse_i]-1)\cdots
(0[\afbse_i]-j_i+1)}{j_i!}\in\Vqz_1.$$
Consequently, $\Vqz=\Vqz_1$.
\end{proof}

\begin{Prop}\label{Vq=Wq}
The set $\fC:=\{A[\bfj]\mid A\in\afThnpm,\,\bfj\in\afmbzn\}$ forms a $\mbq$-basis for $\Vq$.
\end{Prop}
\begin{proof}
For $A\in\afThnpm$, let $\sV_A=\spann_\mbq\{A\br\la\mid \la\in\afmbnn\}$ and $\sW_A=\spann_\mbq\{A[\la]\mid \la\in\afmbnn\}$.
Since, by \cite[6.3.3]{DDF}, the set $\fC$ is linearly independent, it is enough to prove that $\sV_A=\sW_A$ for each $A$.
Fix $A\in\afThnpm$. We now prove $A\br\la\in\sW_A$ and $A[\la]\in\sV_A$ by induction on $\sg(\la)$. If $\sg(\la)=0$ then $A\br\bfl=A[\bfl]\in\sW_A\cap\sV_A$. Now we assume $\sg(\la)>1$ and $A\br\mu\in\sW_A$, $A[\mu]\in\sV_A$ for $\mu\in\afmbnn$ with $\sg(\mu)<\sg(\la)$.
By \ref{formula in Wq} and \ref{Vqz}, we have $0\br\la A\br\bfl=0\br\la A[\bfl]\in\sW_A$.
This, together with \eqref{0{la}A{0}} and the induction hypothesis, implies that $A\br\la\in\sW_A$. On the other hand, by \ref{formula in Wq} we have
\begin{equation*}\label{eq 2-Vq=Wq}
A[\la]=0[\la]A[\bfl]-\sum_{\al\in\afmbnn,\,\al\leq\la
\atop\sg(\al)<\sg(\la)}\bpa{\la}{\al}\ro(A)^{\la-\al}A[\al].
\end{equation*}
Furthermore, by \ref{formula in Vz}(1) and \ref{Vqz}, we have $0[\la]A[\bfl]=0[\la]A\br\bfl\in\sV_A$. Thus, by the induction hypothesis, $A[\la]\in\sV_A$. Consequently, $\sV_A=\sW_A$ for all $A\in\afThnpm$. This finishes the proof.
\end{proof}

\section{Realization of $\sU_\mbz(\h{\frak{gl}}_n)$ and affine Schur-Weyl duality}

In this section, we will prove that $\Vz$ is the realization of $\afuglz$ and use it to prove that the natural surjective algebra homomorphism
$\xi_r:\afuglq\ra\afSrmbq$ remains surjective at the integral level.

By \cite{Yang1} and \cite[6.1.3]{DDF},
there is a unique surjective algebra homomorphism
\begin{equation}\label{xir}
\xi_r:\afuglq\ra\afSrmbq
\end{equation}
such that $\xi_r(\afE_{i,j})=\afE_{i,j}[\bfl,r]$
and $\xi_r(\afE_{i,i})=0[\afbse_i,r]$ for $i\not=j$.
We will see that the maps $\xi_r$ induce an algebra isomorphism $\xi$ from $\afuglq$ to $\Vq$ such that $\xi(\afuglz)=\Vz$.

\begin{Lem}\label{xi}
There is a unique algebra homomorphism
$$\xi:\afuglq\ra\Vq$$
such that $\xi(\afE_{i,j})=\afE_{i,j}[\bfl]$
and $\xi(\afE_{i,i})=0[\afbse_i]$ for $i\not=j$.
\end{Lem}
\begin{proof}
Note that $\afuglq$ has a presentation with
generators $\afE_{i,j}$ ($1\leq i\leq n$, $j\in\mbz$), subject to the following relations:
\begin{itemize}
\item[(a)]
$[\afE_{i,i},\afE_{k,l}]=(\dt_{\bar i,\bar k}-\dt_{\bar i,\bar l})\afE_{k,l}$.
\item[(b)]
$[\afE_{i,j},\afE_{k,l}]=\dt_{\bar j,\bar
k}\afE_{i,l+j-k}-\dt_{\bar l,\bar i}\afE_{k,j+l-i}$ for $i\not=j$ and $k\not=l$.
\end{itemize}
Thus it is enough to prove that
\begin{itemize}
\item[(R1)]
$0[\afbse_i]0[\afbse_k]=0[\afbse_k]0[\afbse_i]$ for all $i,k$;
\item[(R2)]
$0[\afbse_i]\afE_{k,l}[\bfl]
-\afE_{k,l}[\bfl]0[\afbse_i]=(\dt_{\bar i,\bar k}-\dt_{\bar i,\bar l})\afE_{k,l}[\bfl]$ for $k\not=l$;
\item[(R3)]
$\afE_{i,j}[\bfl]\afE_{k,l}[\bfl]-
\afE_{k,l}[\bfl]\afE_{i,j}[\bfl]=\dt_{\bar j,\bar
k}\afE_{i,l+j-k}[\bfl]-\dt_{\bar l,\bar i}\afE_{k,j+l-i}[\bfl]$ for $i\not=j$ and $k\not=l$.
\end{itemize}
For $i,k\in\mbz$, we have
$$0[\afbse_i]0[\afbse_k]=0[\afbse_i+\afbse_k]
=0[\afbse_k]0[\afbse_i],$$
proving (R1).
By definition, for $i\in\mbz$ and $k\not=l\in\mbz$, we have
\begin{equation*}
\begin{split}
0[\afbse_i]\afE_{k,l}[\bfl]-\afE_{k,l}[\bfl]0[\afbse_i]
&=\sum_{\mu\in\afmbzn}(\mu_i+\dt_{\bar i,\bar k})[\afE_{k,l}+\diag(\mu)]-
\sum_{\mu\in\afmbzn}(\mu_i+\dt_{\bar i,\bar l})[\afE_{k,l}+\diag(\mu)]\\
&=(\dt_{\bar i,\bar k}-\dt_{\bar i,\bar l})\afE_{k,l}[\bfl],
\end{split}
\end{equation*}
proving (R2).

It remains to prove (R3). Assume $i\not=j$ and $k\not=l$.
Applying $\xi_r$ to  (b)  yields $$\afE_{i,j}[\bfl,r]\afE_{k,l}[\bfl,r]-
\afE_{k,l}[\bfl,r]\afE_{i,j}[\bfl,r]=\dt_{\bar j,\bar
k}\afE_{i,l+j-k}[\bfl,r]-\dt_{\bar l,\bar i}\afE_{k,j+l-i}[\bfl,r].$$
Multiplying on both sides by $[\diag(\la)]_1$ ($\la\in\afLanr$ and $\la\geq\afbse_i+\afbse_k$) gives the following formula in $\afSrmbq$:
\begin{equation*}
\begin{split}
&[\afE_{i,j}+\diag(\la^{(1)})]_1[\afE_{k,l}+
\diag(\la^{(2)})]_1-
[\afE_{k,l}+
\diag(\la^{(3)})]_1[\afE_{i,j}+
\diag(\la^{(4)})]_1\\
&\quad=\dt_{\bar j,\bar
k}[\afE_{i,l+j-k}+\diag(\la^{(1)})]_1
-\dt_{\bar l,\bar i}[\afE_{k,j+l-i}+\diag(\la^{(3)})]_1,
\end{split}
\end{equation*}
where
$\la^{(1)}=\la-\afbse_i$, $\la^{(2)}=\la-\afbse_i+\afbse_j-\afbse_k$,
$\la^{(3)}=\la-\afbse_k$, $\la^{(4)}=\la-\afbse_k+\afbse_l-\afbse_i$.
Thus by \ref{stabilization property} and the definition of $\sK$, for any $\la\in\afmbzn$, we have in $\hsKq$,
\begin{equation*}
\begin{split}
[\diag(\la)]_1(\afE_{i,j}[\bfl]\afE_{k,l}[\bfl]-
\afE_{k,l}[\bfl]\afE_{i,j}[\bfl])&= [\afE_{i,j}+\diag(\la^{(1)})]_1[\afE_{k,l}+
\diag(\la^{(2)})]_1\\
&\qquad -
[\afE_{k,l}+
\diag(\la^{(3)})]_1[\afE_{i,j}+
\diag(\la^{(4)})]_1\\
&=\dt_{\bar j,\bar
k}[\afE_{i,l+j-k}+\diag(\la^{(1)})]_1
-\dt_{\bar l,\bar i}[\afE_{k,j+l-i}+\diag(\la^{(3)})]_1\\
&=[\diag(\la)]_1(\dt_{\bar j,\bar
k}\afE_{i,l+j-k}[\bfl]-\dt_{\bar l,\bar i}\afE_{k,j+l-i}[\bfl]).
\end{split}
\end{equation*}
This implies that,
\begin{equation*}
\begin{split}
\afE_{i,j}[\bfl]\afE_{k,l}[\bfl]-
\afE_{k,l}[\bfl]\afE_{i,j}[\bfl]&
=\sum_{\la\in\afmbzn}[\diag(\la)]_1(\afE_{i,j}[\bfl]\afE_{k,l}[\bfl]-
\afE_{k,l}[\bfl]\afE_{i,j}[\bfl])\\
&=\sum_{\la\in\afmbzn}[\diag(\la)]_1(\dt_{\bar j,\bar
k}\afE_{i,l+j-k}[\bfl]-\dt_{\bar l,\bar i}\afE_{k,j+l-i}[\bfl])\\
&=\dt_{\bar j,\bar
k}\afE_{i,l+j-k}[\bfl]-\dt_{\bar l,\bar i}\afE_{k,j+l-i}[\bfl],
\end{split}
\end{equation*}
proving (R3).
\end{proof}

We can now prove that $\Vz$ gives a BLM realization of $\afuglz$, which is the affine version of \cite[5.7]{BLM}.


\begin{Thm}\label{realization of afuglz}
$(1)$ The algebra homomorphism $\xi:\afuglq\ra\Vq$ defined in \ref{xi} is an algebra isomorphism.

$(2)$ $\afuglz$ is a $\mbz$-subalgebra of $\afuglq$ and the restriction of $\xi$ to $\afuglz$ gives a $\mbz$-algebra isomorphism $\xi:\afuglz\ra\Vz$.
\end{Thm}
\begin{proof}
We first prove (1). Let $\sL^+=\{(i,j)\mid
1\leq i\leq n,\ j\in\mbz,\ i<j\}$ and $\sL^-=\{(i,j)\mid
1\leq i\leq n,\ j\in\mbz,\ i>j\}$.
By \ref{iotap}(1), \ref{Vzp-Vzm}(1) and \eqref{tau(Vzp)}, the set
$$\bigg\{\prod_{(i,j)\in\sL^+}(\afE_{i,j}[\bfl])^{a_{i,j}}
\,\bigg|\, A\in\afThnp\bigg\}\bigg(resp., \bigg\{\prod_{(i,j)\in\sL^-}(\afE_{i,j}[\bfl])^{a_{i,j}}
\,\bigg|\, A\in\afThnm\bigg\}\bigg)$$
forms a $\mbq$-basis for $\Vqp$(resp., $\Vqm$), where
the products are taken with respect to a fixed total order on $\sL^+$ (resp., $\sL^-$). This, together with \ref{tri decomposition} and  \ref{Vqz}, implies that the set
$$\bigg\{\prod_{(i,j)\in\sL^+}(\afE_{i,j}[\bfl])^{a_{i,j}}
0[\bfj]
\prod_{(i,j)\in\sL^-}(\afE_{i,j}[\bfl])^{a_{i,j}}
\,\bigg|\, A\in\afThnpm,\,\bfj\in\afmbnn\bigg\}$$
forms a $\mbq$-basis for $\Vq$. Thus $\xi$ sends a PBW-basis of $\afuglq$ to a basis of $\Vq$. Consequently, $\xi$ is an algebra isomorphism.

To see (2), by (1) and \ref{Vz-subalgebra}(1), it is enough to prove $\xi(\afuglz)=\Vz$. Recall the algebra homomorphism
$\iota^+:\Hall_\mbq\ra\afuglq$ defined in \ref{iotap}  and the algebra isomorphism $\thp:\Hallz\ra\Vzp$ described in  \ref{Vzp-Vzm}.
The map $\thp$  induces an injective algebra isomorphism $\Hall_\mbq\ra\Vq$, which is also denoted by $\thp$. Since, by \ref{iotap}, $\Hall_\mbq$ is generated by $\afE_{i,j}$ for $i<j$, and since $$\xi\circ\iota^+(u_{{\afE_{i,j},1}})=\xi(\afE_{i,j})
=\afE_{i,j}[\bfl]=\thp(u_{{\afE_{i,j},1}})$$
for $i<j$,
we conclude that $\xi\circ\iota^+=\thp$. So $\Vzp=\thp(\Hall_\mbz)=\xi\circ\iota^+(\Hall_\mbz)=\xi(\afuglzp)$.
Similarly we have $\Vzm=\xi(\afuglzm)$. Furthermore, since $$\xi\left(\prod_{1\leq i\leq n}\bpa{\afE_{i,i}}{\la_i}\right)=0\br\la,$$
we have $\Vzz=\xi(\afuglzz)$. Thus, by \ref{tri decomposition} we conclude that
$$\Vz=\Vzp\Vzz\Vzm=
\xi(\afuglzp\afuglzz\afuglzm)=\xi(\afuglz),$$
proving (2).
\end{proof}

\begin{Coro}\label{Hopf Z-subalgebra}
$\afuglz$ is a $\mbz$-Hopf subalgebra of $\afuglq$ with comultiplication given by
\begin{equation*}
\begin{split}
\Dt(\iota^+(u_{\la,1}))
&=\sum_{\la^{(1)},\la^{(2)}\in\afmbnn\atop\la=\la^{(1)}+\la^{(2)}}
\iota^+(u_{\la^{(1)}})\ot\iota^+(u_{\la^{(2)}})\\
\Dt(\iota^-(u_{\la,1}))
&=\sum_{\la^{(1)},\la^{(2)}\in\afmbnn\atop\la=\la^{(1)}+\la^{(2)}}
\iota^-(u_{\la^{(1)}})\ot\iota^-(u_{\la^{(2)}})\\
\Dt\left(\bpa{\afE_{i,i}}{t}\right)&=\sum_{0\leq j\leq t}\bpa{\afE_{i,i}}{j}\ot\bpa{\afE_{i,i}}{t-j}
\end{split}
\end{equation*}
for $\la\in\afmbnn$, $1\leq i\leq n$ and $t\in\mbn$, where $u_{\la,1}=u_{[S_\la]}\ot 1$.
\end{Coro}
\begin{proof}
By \eqref{eq 2-Hallp}, \eqref{eq 3-Hallp}, \eqref{eq 4-Hallp}, \eqref{eq Hopf algebra isomorphism} and \ref{Hopf algebra isomorphism}, $\afuglzp$ and $\afuglzz$ are $\mbz$-Hopf subalgebras of $\afuglq$. Clearly, there is a natural algebra anti-isomorphism
$$\Phi:\afuglq\ra\afuglq(\afE_{i,j}\map\afE_{j,i}\ \forall i,j).$$
Since $\Phi(\afuglzp)=\afuglzm$, and $\Phi$ preserves comultiplication and antipode, $\afuglzm$ is also a $\mbz$-Hopf subalgebra of $\afuglq$. The proof is completed.
\end{proof}

Finally, we will prove that $\xi_r:\afuglz\ra\afSrmbz$ is surjective.
By restriction, the algebra homomorphism $\hzr:\hsKq\ra\afSrmbq$ defined in \eqref{hzr} induces an algebra
homomorphism
\begin{equation}\label{zr}
\zr:=\hzr|_{\Vq}:\Vq\ra\afSrmbq.
\end{equation}
Clearly,  $\zr(A\br\la)=A\br{\la,r}$
for $A\in\afThnpm$ and $\la\in\afmbnn$.
\begin{Lem}\label{zr(Vz)}
The  algebra
homomorphism $$\zr:\Vq\ra\afSrmbq$$ is surjective and we have $\zr(\Vz)=\afSrmbz$.
\end{Lem}
\begin{proof}
Since $0\br{\la,r}=[\diag(\la)]$ for $\la\in\afLanr$, we have $$\zr(\Vzz)=\spann_{\mbz}\{0\br{\la,r}\mid\la\in\afmbnn\}
=\spann_{\mbz}\{[\diag(\la)]\mid\la\in\afLanr\}.$$
Thus, by \ref{tri decomposition} and \ref{tri-affine Schur algebras},
\begin{equation*}
\begin{split}
\zr(\Vz)&=\spann_\mbz\{A^+\br{\bfl,r}0\br{\la,r}
A^-\br{\bfl,r}\mid A\in\afThnpm,\,\la\in\afmbnn\}\\
&=\spann_\mbz\{A^+\br{\bfl,r}[\diag(\la)]
A^-\br{\bfl,r}\mid A\in\afThnpm,\,\la\in\afLanr\}\\
&=\afSrmbz,
\end{split}
\end{equation*}
proving the assertion.
\end{proof}

\begin{Thm}\label{affine Schur-Weyl duality}
The restriction of $\xi_r$ to $\afuglz$ gives a surjective $\mbz$-algebra homomorphism
$$\xi_r:\afuglz\twoheadrightarrow\afSrmbz.$$
\end{Thm}
\begin{proof}
Clearly, $\zr\circ\xi(\afE_{i,j})=\zr(\afE_{i,j}[\bfl])
=\afE_{i,j}[\bfl,r]=
\xi_r(\afE_{i,j})$ for any $i\not=j\in\mbz$ and $\zr\circ\xi(\afE_{i,i})=\zr(0[\afbse_i])=0[\afbse_i,r]=
\xi_r(\afE_{i,i})$. Thus we have the following
commutative diagram:
$$\xymatrix{
\afuglq \ar[rr]^{\xi}\ar[d]_{\xi_r}  &  & \Vq \ar [dll]^{\zr} \\
\afSrmbq & &
}$$
It follows from \ref{realization of afuglz} and \ref{zr(Vz)} that $\xi_r(\afuglz)=\zr\circ\xi(\afuglz)=\zr(\Vz)=\afSrmbz$.
\end{proof}

Let $k$ be a field. We denote $\afuglk=\afuglz\ot k$ and
$\afSrk=\afSrmbz\ot k$.
\begin{Coro}
For any field $k$, the algebra homomorphism
$$\xi_r\ot id:\afuglk\twoheadrightarrow\afSrk
\cong\End_{k\affSr}\left(\bop_{\la\in\afLanr}k\ul{{\fS}_\la}\affSr
\right)$$
is surjective.
\end{Coro}

\end{document}